\documentclass[journal]{IEEEtran}

\usepackage{amsmath}
\usepackage{float}
\usepackage{cite}
\usepackage{enumerate}
\usepackage{graphicx}
\usepackage{amsfonts}
\usepackage{amssymb}
\usepackage{mathtools}
\usepackage{xcolor}
\usepackage{enumerate}
\usepackage{xspace}
\usepackage{xparse}
\usepackage{anyfontsize}
\usepackage{scalerel}
\usepackage[thmmarks, amsmath, thref, hyperref]{ntheorem}
\usepackage[thinc]{esdiff}

\usepackage{balance}
\usepackage{enumitem}
\usepackage{mathrsfs} 
\usepackage{physics}
\usepackage{stfloats}% <-- added
\usepackage{blindtext}
\usepackage{epstopdf}
\usepackage[labelformat=simple]{subcaption}

\renewcommand{\qed}{\hfill\blacksquare}

\newlist{alphalist}{enumerate}{1}
\setlist[alphalist,1]{label={C\arabic*}-}

\allowdisplaybreaks

\newcommand{\jw}{\mathbf{j}\omega}

\renewcommand{\j}{\mathbf{j}}

\newcommand{\ie}{{i.e. }}
\newcommand{ \RHi} {\mathbb{RH}_{\infty}}
\newcommand{\C}{\mathbf{C}}
\newcommand{\R}{\mathbf{R}}
\newcommand{ \jR} {\j\overline{\R}}
\newcommand{\w}{\omega}
\newcommand{\etc}{\emph{etc}.}
\newcommand{\QC}{\mbox{$\bf{QC}$}}
\newcommand{\QCinv}{\mbox{$\overline{\bf{QC}}$}}

\newtheorem{theorem}{Theorem}

\newtheorem{assumption}{Assumption}
\newtheorem{lemma}{Lemma}

\newtheorem{remark}{Remark}

\newtheorem*{proof}{Proof}

\theoremstyle{nonumberbreak}
\theoremheaderfont{\bfseries}
\theorembodyfont{\normLalfont}
\theoremsymbol{\ensuremath\square}

\newcommand{\kkl}{\textcolor{black}}
\newcommand{\khl}{\textcolor{black}}

\newcommand{\kl}{\textcolor{black}}
\newcommand{\il}{\textcolor{black}}
\newcommand{\icl}{\textcolor{black}}
\newcommand{\ic}{\textcolor{black}}
\newcommand{\ilc}{\textcolor{black}}
\newcommand{\ilf}{\textcolor{black}}
\newcommand{\ilff}{\textcolor{black}}
\newcommand{\iclf}{\textcolor{black}}

\newcommand{\klRev}{\textcolor{black}}
\newcommand{\icr}{\textcolor{black}}

\usepackage[prependcaption,colorinlistoftodos]{todonotes}

\theoremstyle{nonumberbreak}
\theoremheaderfont{\bfseries}
\theorembodyfont{\normLalfont}
\theoremsymbol{\ensuremath\square}

\newcommand{\ir}{\textcolor{black}}
\newcommand{\irr}{\textcolor{black}}

\makeatletter
\def\thanks#1{\protected@xdef\@thanks{\@thanks
		\protect\footnotetext{#1}}}
\makeatother

\begin{document}

\title{	Decentralized  Stability Conditions for DC Microgrids: Beyond Passivity  Approaches\thanks{This work was supported by ERC starting grant 679774. } \thanks{The authors are  with the Department of Engineering, University of Cambridge,Trumpington Street, Cambridge CB2 1PZ, United Kingdom.  Emails: \{kl507,jdw69, yo259, icl20\}@cam.ac.uk.}}

\author{Khaled Laib, Jeremy  Watson, Yemi Ojo  and Ioannis Lestas}

\date{}
\maketitle

	\begin{abstract} 
		We consider the problem of ensuring stability in a DC microgrid by means of decentralized conditions. Such conditions are derived which are formulated as input-output properties of locally defined subsystems. These follow from various decompositions of the microgrid and corresponding properties of the resulting representations. It is shown that these stability conditions can be combined together by means of appropriate homotopy arguments, thus reducing the conservatism relative to more conventional decentralized approaches that often rely on a passivation of the bus dynamics.
		Examples  are presented to demonstrate the efficiency and the applicability of the results derived. 	
    \end{abstract}

	   \section{Introduction}	
	 
	   The increasing integration of renewable energy in recent years has strengthened the interest in microgrids. Compared to \iclf{AC} %Alternating Current (AC) 
	   microgrids,
	   %Direct Current~(DC) 
	   \iclf{DC microgrids} have been recognized as a natural and simple solution to integrate renewable energy,   see~\cite{JMLJ:13}.
	   For instance, DC~microgrids allow   connecting   DC components  directly for a simple integration of renewable \iclf{generation} and  storage units. Moreover, in \il{addition to} being compatible  with modern consumer loads, DC~microgrids  allow to reduce  unnecessary power conversion losses. Thus, DC~microgrids  have become an attractive option not only  \il{for providing} support to remote communities,  but in many  other applications \iclf{such as} mobile grids on ships, aircrafts, \iclf{and trains~\cite{EMM:15}}.

	   A DC~microgrid is a  power network that	consists of small subsystems (generation units, storage units, flexible loads, \etc)  interconnected \iclf{via}   power lines.   
	   \il{A key \il{requirement} in a DC microgrid  
	   	is to \il{ensure 
	   		stability of the network when decentralized feedback control mechanisms are used for voltage/current regulation~\cite{MSTKFLG:17}.}}

	   \il{Two main features} can lead  to network instability: \il{interaction between converters} and load type, see~\cite{MSTKFLG:17,DLVG:16}.

	   \il{In particular, each DC-DC converter} %which
	   is  typically  designed to guarantee good stability margins and achieve certain performance levels when operating in a stand-alone condition. However, when interconnecting the different  DC-DC converters in the network, their interaction   can \iclf{affect} the overall performance and \iclf{even} lead to \iclf{instability}.

	   \il{Furthermore, some} DC loads have a destabilizing effect. For instance, in contrast to constant \il{impedance loads %(Z)
	   	and constant current loads, % (I),
	   	which normally do not induce stability degradation, constant power loads %(P)
	   	can lead to network instability} due to their negative impedance characteristics.

	   \subsection{Literature review}
	   
	   Various decentralized (linear and nonlinear) control strategies   have been  proposed in the literature to ensure proper functioning of DC~microgrids: \il{droop control schemes}~\cite{GVMVC:11,ZhD:15}, \il{line-independent approaches}~\cite{TRF:18}, cooperative \il{schemes}~\cite{NMDL:15} , passivity-based~\cite{GLH:15,FCS:21},  Lyapunov-based~\cite{IDSBGP:18}, backstepping~\cite{RMOHMM:18}, \iclf{and} sliding-mode \il{control} \iclf{schemes}~\cite{CTDPCFVDS:19}.
	   
	   However, the aforementioned controllers do not handle situations where  constant power loads are present and some of them \ilff{do not incorporate} %neglect
	   the power  line dynamics  and consider them purely resistive.

	   To  address the voltage destabilizing effect of constant power loads, various controllers  have been proposed in the literature to handle general nonlinear ZIP loads, \il{i.e.
	   	parallel combination of constant impedance, current, and power load respectively (denoted as Z, I, P respectively).}

	   In~\cite{PWD:18}, the authors propose a   consensus algorithm   guaranteeing power consensus in a network  with~ZIP  loads. In~\cite{CLKKS:19}, the authors  propose a nonlinear passivity based voltage controller with some robustness  with respect to constant  ZIP loads.    In~\cite{NSMMFT:20}, the authors propose  a  linear state feedback voltage controller to passivate the generation units and  the~ZIP loads connected to it.

	   \il{Existing results that incorporate dynamic line models rely primarily on a passivation of the bus dynamics so as to achieve stability in general network topologies. The latter implies that there are restrictions on the amount of constant power loads that can be incorporated, as these have a non-passive behaviour. Furthermore, many classical converter control architectures, such as ones based on double-loop implementations that provide also current control capabilities, are often hard to passivate in practical designs.
	   	Therefore the development of methodologies that can reduce the conservatism in the stability conditions imposed, while at the same having stability guarantees in general network topologies is an important problem of practical relevance.}

	   \subsection{Main contributions}
	   \il{The objective of this  paper is  to derive   decentralized conditions  through which  stability  of the DC~microgrid  can be established, i.e. conditions on locally defined subsystems. The microgrid representation plays a central role in this context as the notion of a subsystem is not unique. In particular,
	   	any derived stability conditions are inevitably going to
	   	be only sufficient when these are decentralized. Therefore,
	   	different representations of what constitutes a subsystem within the network can
	   	lead to stability results with varying conservatism.}

	   \il{A main idea of our analysis is to consider different decentralized input-output conditions based on various decompositions of a DC network and then combine them together by means of appropriate homotopy arguments. This allows to exploit on the one hand the natural passivity properties of the lines in frequency ranges where the coupling between the buses is high, and also exploit other input-output conditions  in frequency ranges where  there are significant deviations from \ilc{passivity} %(e.g. higher frequency ranges),
	   	by additionally taking into account the \ilc{strength} of the coupling between buses.}
	   
	   \il{A key contribution of the proposed approach is that it allows to reduce the conservatism in the design relative to more conventional methodologies, such as ones that rely on a passivation of bus dynamics. In particular, it enables larger amounts of constant power loads to be incorporated while guaranteeing stability of the network, and also allows to establish stability for wider classes of practically relevant control architectures.}
	   
	   \il{It should also be noted that the input-output approach adopted allows to consider broad classes of microgrid models involving higher order converter models and line dynamics. Related practical examples  will also be discussed within the paper to demonstrate the significance of the results presented.}

%%%%%%%%%%%%%%%%%%%%%%%%%%%%%%%%%%%%%%%%%%%%%%%%%%%%%%%%%%%%%%%%%%%%%%%%%%%%%%%%

\subsection{Paper outline}	

{The paper} is structured as follows.   Section~\ref{sec:prob_form} presents some graph theory elements,  the DC~microgrid model and the {problem setting.}
Section~\ref{sec:main_results}  presents the main results of the paper while  Section~\ref{sec:num_exp} presents   numerical examples.  Conclusions   are drawn in Section~\ref{sec:concl}.

\subsection{Symbols and notations}
The sets of real and complex numbers are denoted by~$\R$  and $\C$ respectively.  The extended real line $\left[ -\infty,+\infty\right]$ is denoted $\overline{\R}$ and  $\overline{\R}_+$ is the  set of  real positive numbers including~0  and $+\infty$. The imaginary axis is denoted by~$\jR$ where $\j=\sqrt{-1}$.
The right  half-{plane} including the imaginary axis  is denoted by~${\C}_{+}$ and its closure is denoted by~$\overline{\C}_{+}$.
\ilc{The space $\mathscr{L}_{2}^n[0,\infty)$ is the set of %$\R^n$ valued
	signals $f:[0,\infty)\rightarrow\R^n$ %on $[0,\infty)$ \
	that have finite energy $\int_0^\infty \|f(t)\|^2\text{d}t$, and  $\mathscr{H}^n_2$ is the set of \ilc{functions that are  Laplace transforms of   signals in $\mathscr{L}_{2}^n[0,\infty)$}.}
The set 	   $\RHi^{p\times q}$ is the set  of ${p\times q}$ real rational   transfer functions  without poles in~$\overline{\C}_+$. 	
For a	matrix $F \in \C^{n \times n}$ its transpose and conjugate transpose are denoted by   $F^{\top}$ and~$F^*$ respectively.  \ilf{For a	matrix $F \in \C^{m \times n}$,  $\|F\|_\infty$, $\rho(F)$ denote its induced $\infty$-norm and its spectral radius respectively.}
The  identity matrix is denoted by $I$.% and its {dimensions} are deduced from the context.
The Kronecker product  of two  matrices $F_i$ and $F_j$ is denoted by $F_i \otimes F_j$.
The direct sum
of matrices $F_i$ with $i=1,\cdots,n$  is denoted by $\oplus_{i=1}^{n} F_i$.
Finally, in order to ease the notation, for  matrices $X, \Pi_{11}, \Pi_{12}, \Pi_{22}$ with compatible dimensions and $\epsilon\geq0$, we use  $X\in \QC\left(\Pi,\epsilon  \right)$, with $\Pi= 	\begin{pmatrix} \Pi_{11}&  \Pi_{12}\\
\Pi_{12}^*  & \Pi_{22}	\end{pmatrix}$, to denote 	
\begin{equation}\label{def:QC}	\begin{pmatrix} X \\ I  \end{pmatrix}^*	
~~\Pi~~
\begin{pmatrix} X \\ I  \end{pmatrix}\geq \epsilon X^*X 	  \end{equation}
and $X\in \QCinv\left(\Pi,\bar \epsilon  \right)$ to denote
$$	\begin{pmatrix} I\\ - X  \end{pmatrix}^*
~~\Pi~~
\begin{pmatrix} I\\ - X    \end{pmatrix} \leq - \bar \epsilon  I .$$

\section{Network models and problem {setting}} \label{sec:prob_form}
\subsection{Algebraic graph theory and microgrid signals}
\klRev{The DC microgrid  is a power system   that comprises of~$n_{b}$ buses   and~$n_{\ell}$  {power lines}.
We assume  that each bus includes a  \mbox{DC-DC} converter with its controller, and a load connected to it. 	Note that even if loads are located elsewhere, 		they can be mapped to a point of common coupling (PCC) using Kron reduction, see~\cite{DoB:13}. Fig.~\ref{fig:MGkmodel} is a representative diagram of a microgrid with six buses and six lines while Fig.~\ref{fig:BusScheme} gives \irr{a schematic} of the electrical connection  of the $j^{th}$  bus.}

\begin{figure}
\centering
\includegraphics[width=1\linewidth]{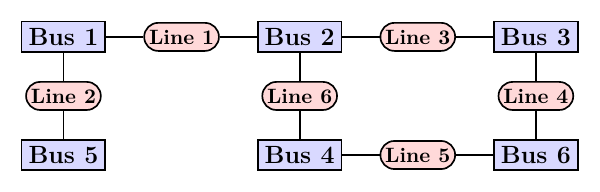}
\caption{\ir{A representative diagram of \ir{a six} bus microgrid.}}
\label{fig:MGkmodel}
\end{figure}

\begin{figure}
\centering
\includegraphics[width=1\linewidth]{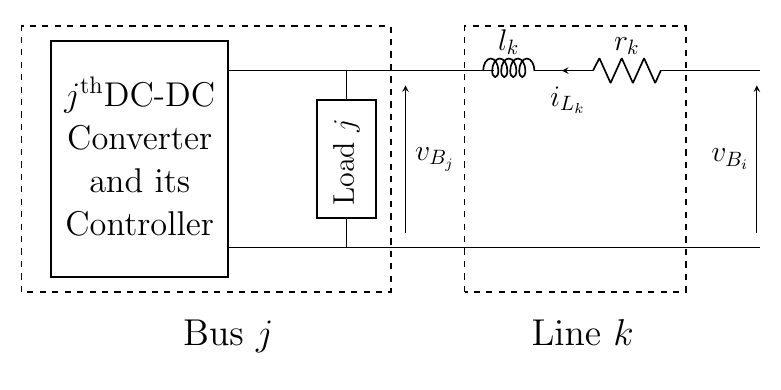}
\caption{\ir{Diagram representing the $j^{th}$  bus (\mbox{DC-DC} converter with its controller, and a load)  \ir{connected to} the  $k^{th}$  line.}}
\label{fig:BusScheme}
\end{figure}

We represent this microgrid as a connected  graph $(\mathcal{N},\mathcal{E})$ where $\mathcal{N}=\left\lbrace 1,2,\cdots, n_{b}\right\rbrace $ is the set of nodes~(buses)
and  $\mathcal{E}=\left\lbrace 1,\cdots, n_{\ell}\right\rbrace \subseteq \mathcal{N}\times \mathcal{N}$ is the set of edges~(lines).
A direction is assigned to each edge which can be arbitrarily chosen.
The  corresponding $n_b \times n_\ell$ incidence matrix is denoted by~$\mathcal{A}$
and it is given by
\begin{equation*}
\mathcal{A}_{jk}=
\left\lbrace
\begin{matrix}
&+1 & \text{if edge $k$ leaves bus $j$,}\\
&-1 & \text{if edge $k$ enters bus $j$,}\\
&~~0 & \text{otherwise.~~~~~~~~~~~~~~}\\
\end{matrix}
\right.
\end{equation*} 	
For each node $j\in \mathcal{N} $,  the set  $\mathcal{E}_j=\left\lbrace k\in \mathcal{E}:  \mathcal{A}_{jk} \neq 0\right\rbrace$ is the set  of   edges  connected to  node $j$.

Given the aforementioned microgrid settings, we define the following  signals.

\begin{itemize}
\item 	The input current {at}   bus $j$ is denoted  by~$i_{B_j}(t)$ and  the bus output
{voltage} is  denoted by~$v_{B_j}(t)$.		
\item     The current through   line  $k$ is denoted by~$i_{L_k}(t)$ {(this denotes the current with the same direction as that of the edge $k$)}.
\item 	The {vector of all bus voltages}  and the vector of all line currents  are denoted  by~$v_B(t)=\left(  v_{B_j}(t)\right) _{j\in \mathcal{N}} $ and $  i_L(t)  =\left(i_{L_k}(t)\right)_{k\in \mathcal{E}}$ respectively.
\end{itemize}

\subsection{Line dynamics}

\il{We consider power lines modeled as RL components.}		
\klRev{The lines of the DC microgrid connect the buses and allow power to be transferred from one bus to another and across the microgrid as a whole. The \ir{current}  flowing across a line is determined by the difference between the voltages at each bus to which
the line is connected.}  
By applying  Kirchoff’s voltage law  on the $k^{th}$ {power line}, with $k\in \mathcal{E}$, we obtain
\begin{equation} \hspace{-0.1cm}
\diff{i_{L_k}(t)}{t} =    \dfrac{-r_k}{l_k}
i_{L_k}(t) + \dfrac{1}{l_k} \delta V_{L_k} (t)
\label{eq:Line_dynamics}
\end{equation}
where $r_k>0$ and $l_k>0$ are the resistance and the inductance of $k^{th}$ line   and $\delta V_{L_k}(t)$  is the voltage difference  {across line $L_k$.} 

\il{
It should be noted though that the stability results that will be presented also hold
if the transfer function from~$\delta V_{L_k}$  to~$i_{L_k}$ is any strictly positive real function.
This allows, \il{for example,} to consider also more advanced line models that include capacitances or distributed parameter models as in \cite{WOLL:20}.
Extensions to cases where this transfer function is positive real rather than strictly positive real, will also be considered in Theorem~\ref{thm:Main_result_2}.
}

\subsection{Bus dynamics}

\klRev{As already mentioned, each bus includes a DC-DC converter with its own controller, and a load connected to it. The voltage at each bus is controlled via the DC-DC converter using local information only (i.e. the states of the bus and the input current from the microgrid).}

We consider 	 a general bus model  to account for  a broad class of DC-DC converters, controllers and loads. 	
DC-DC converters (buck, boost, buck-boost, SEPIC, \etc) are composed of three main stages: DC stage (battery stage), switching stage  and DC output stage (output-voltage stage).

\ir{Average models, i.e. models described by continuous ODEs,  are commonly used in the literature to describe the converter dynamics so as to carry out stability analysis and control design. These are justified under the following assumption.}

~

\begin{assumption}

\klRev{	The switching \ir{of the converter} is performed at a frequency much higher than the \ir{timescale of its control policies}.} 
\label{assump:converter}
\end{assumption}

\ir{Thereafter, average models will be used throughout the paper to model the converter dynamics.}

The converter dynamics as well as the controller dynamics vary depending on the DC~microgrid voltage level (low, medium and high), model complexity, control strategy, \etc. Therefore,   the converter and the controller  dynamics will be kept {in a general representation for analysis and particular implementations will be discussed} %presented and illustrated
in the  {examples of}  Section~\ref{sec:num_exp}. For the \ilf{loads}, we consider a {general}  %nonlinear
ZIP load model {which includes constant impedance, constant current and constant power loads}.

A general form  of the $j^{th}$ bus  dynamics {can}
be described as  single-input single-output dynamical system with $i_{B_j}(t)$ as input  and $v_{B_j}(t)$  as output. These dynamics are {represented}
as follows  
\begin{equation}
\left\lbrace
\begin{matrix}
&\diff{x_{B_j}(t)}{t} =
f_{B_j}\big( x_{B_j}(t),    i_{B_j}(t)\big) \\
& ~~  v_{B_j}(t) =  g_{B_j}\big( x_{B_j}(t),    i_{B_j}(t)\big)
\end{matrix}
\right.   \label{eq:bus_dynmics}
\end{equation}
where
$x_{B_j}(t) \in \R^{{n_{x_{B_j}}}}$ is the state %space variable
vector at each bus ({includes} converter, controller, and load states),  $f_{B_j}$ and $g_{B_j}$ {are functions of the form} $f_{B_j}  :\R^{{n_{x_{B_j}}}} \times \R \rightarrow \R^{{n_{x_{B_j}}}} $ and $g_{B_j}  :\R^{{n_{x_{B_j}}}} \times \R \rightarrow \R$.

\subsection{Microgrid small-signal model}

The bus model~\eqref{eq:bus_dynmics} is in general nonlinear due to the converter  and load dynamics even when  considering linear controllers;  hence the    microgrid model is also~nonlinear.		
Equilibrium points can be found by setting the  time derivatives in~\eqref{eq:Line_dynamics}-\eqref{eq:bus_dynmics}  to zero and then {solving} the resulting system of equations.

\klRev{Finding \ir{\irr{equilibrium points}} in a power grid is the well-known power-flow
problem, which has been studied in depth \ir{e.g. \cite{gar:18,mon:20}.}
\icr{Load / generation fluctuations result in deviations 	from \irr{a nominal} operating point, however when these are small, which is usually the case under normal operating conditions, a small signal analysis can be used for stability analysis and control design.}} %in practice.}

\klRev{Thereafter, a linearization \icr{is} performed  using an obtained  equilibrium.	{For this purpose, we require   the following assumption.	}}

~ 	
\begin{assumption} The system~\eqref{eq:Line_dynamics}-\eqref{eq:bus_dynmics} admits an equilibrium. Moreover, the vector functions   $f_{B_j}$  and the  functions~$g_{B_j}$ in~\eqref{eq:bus_dynmics} are  Lipschitz around the considered~equilibrium \mbox{of~\eqref{eq:Line_dynamics}-\eqref{eq:bus_dynmics}}.
\label{assump:linearization}
\end{assumption}

Under Assumption~\ref{assump:linearization}, the system~\eqref{eq:Line_dynamics}-\eqref{eq:bus_dynmics} can be linearized
about the equilibrium being considered. In order to analyze this
linearization, let  $\overline{q}(t)=q(t)-q^{eq}$  denote the deviation of
any quantity~$q(t)$  from its equilibrium value~$q^{eq}$. We denote the microgrid equilibrium by  $q^{eq}$   given by 
\begin{equation}
q^{eq}=  \begin{pmatrix}
\big( i_L^{eq}\big)^{\top}, & \big(v_B^{eq}\big)^{\top}, &
\big(x_B^{eq}\big)^{\top}
\end{pmatrix}^{\top}
\label{eq:equilibrium}
\end{equation}
with $x_B^{eq}=( x^{eq}_{B_j})_{j \in \mathcal{N}}$. 	 
Finally, we adopt an input-output representation of the  small-signal model of  microgrid~\eqref{eq:Line_dynamics}-\eqref{eq:bus_dynmics},
and we introduce the following two sets  of transfer  functions

\begin{itemize}
\item  $L_k(s)$  is the  transfer  function   of the line  dynamics~\eqref{eq:Line_dynamics},   from the input $\overline{\delta V_{L_k}}(t)  $ to the output $ \overline{i_{L_k}}(t)$;		

\item  $B_{j}(s)$  is the  transfer  function  {of} the linearized version of the bus dynamics~\eqref{eq:bus_dynmics}  from  the input $\overline{i_{B_j}}(t)$ to the output $ \overline{v_{B_j}}(t)$.
\end{itemize}

Note that  the different $L_k(s)$, obtained from~\eqref{eq:Line_dynamics},  are   in   $\RHi$   as   $-r_kl_k^{-1}<0$.
\kkl{ We will derive conditions on the 	frequency response of locally defined subsystems   under which
stability of the power system~\eqref{eq:Line_dynamics}-\eqref{eq:bus_dynmics} about the equilibrium~\eqref{eq:equilibrium}    is guaranteed. To do this, we introduce the following \ilc{assumption.}}

\kkl{\begin{assumption}\label{assump:stab_bus}
For each bus $j$,    $B_{j}(s)\in\RHi$ with a  stabilizable and detectable state-space realization\footnote{\kkl{The stabilization and the detectability assumptions ensure that there are no \ilc{pole/zero cancellations in ${\C}_{+}$}  in the transfer function}.}.
\end{assumption}}

Let   $ \overline{v_B}(t)=\left(  \overline{v_{B_j}}(t)\right) _{j\in \mathcal{N}} $ and  $ \overline{i_L}(t)=\left(  \overline{i_{L_k}}(t)\right) _{k\in \mathcal{E}} $; then $\overline{\delta v_{L_k}}(t)= \left( a^c_k\right)^{\top}      \overline{v_B}(t)$ and $\overline{i_{B_j}}(t)=- ~a^r_{j}~  \overline{i_{L}}(t)$  where~$a^c_k $ and~$~a^r_{j}~ $ are  the $k^{th}$ column and  the $j^{th}$ row of the incidence matrix~$\mathcal{A}$ respectively. 	\\	

\begin{figure}[t]
\centering
\includegraphics[width=1\linewidth]{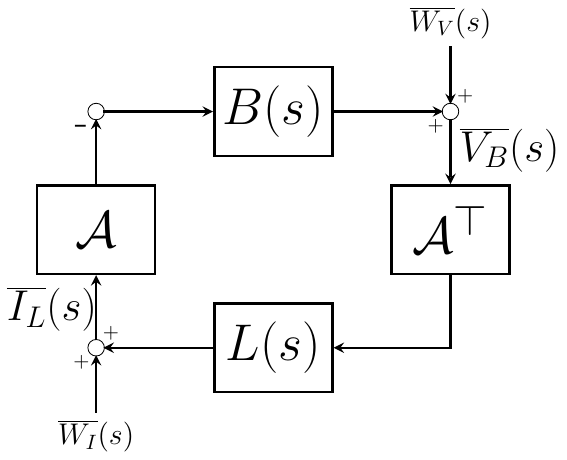}
\caption{\khl{Block diagram representing the small-signal \ilf{model}~\eqref{eq:bipartite_eqt}}.}
\label{fig:networkmodel}
\end{figure}		
\kkl{The  small-signal model of the microgrid~\eqref{eq:Line_dynamics}-\eqref{eq:bus_dynmics}  can be represented as a negative feedback interconnection of \mbox{input-output} systems
as 	illustrated in  Fig.~\ref{fig:networkmodel}	  
and is  given by 
\begin{equation}\left\lbrace
\begin{matrix}
\overline{V_B}(s)= -B(s)~\mathcal{A}~~ \overline{I_L}(s)+\overline{W_V}(s)\\%\Theta_B(s)~\overline{x_B}(0) \\
\overline{I_L}(s)= ~~L(s)~ \mathcal{A}^{\top}~ \overline{V_B}(s) +\overline{W_I}(s)\\%L(s)K~\overline{i_L}(0) ~~~~~
\end{matrix} \right.
\label{eq:bipartite_eqt}
\end{equation}
where $ B(s)=\oplus_{j=1}^{n_b}  B_{j}(s)$ and $L(s)=\oplus_{k=1}^{n_\ell}  L_k(s)$
with $B_j(s) \in \RHi$  and  $L_k(s) \in \RHi$.
$\overline{W_V}(s)$ and $\overline{W_I}(s)$ \ilc{are signals associated  with the initial conditions. In particular, these} are given, \ilc{respectively,} by   $\overline{W_V}(s)=\Theta_B(s)~ \overline{x_B}(0)$ and $\overline{W_I}(s)=L(s) \left( \oplus_{l=1}^{n_\ell}l_k\right)~ \overline{i_L}(0)$, \ilc{where}
$\overline{x_B}(0)$ and $\overline{i_L}(0)$ are the initial \ilc{values} %conditions
of   $( x_{B_j}(t))_{j \in \mathcal{N}}$ and $\overline{i_L}(t)$,  \iclf{respectively,} and
\begin{equation}
\Theta_B(s)= \left(\oplus_{j=1}^{n_b} C_{B_j} \right) \left(sI-\oplus_{j=1}^{n_b} A_{B_j} \right)^{-1}
\label{eq:Theta_B}
\end{equation} 
where $C_{B_j}$ and $A_{B_j}$ are  the output matrix and  the state matrix in the state-space representation of $B_j(s)$.
Note that  $\Theta_B(s) $ is in $ \RHi$ (from Assumption~\ref{assump:stab_bus}), hence both  signals $\overline{W_V}(s)$ and $\overline{W_I}(s)$ are in~$\mathscr{H}_2$.}

\section{Main results}\label{sec:main_results}

We derive in this section sufficient conditions for the local  stability of the    DC~microgrid~\eqref{eq:Line_dynamics}-\eqref{eq:bus_dynmics} \il{which are formulated as decentralized} %based on %decentralized
input-output conditions on  different subsystems.  	As mentioned in the introduction,
\il{the notion of a subsystem is not unique in a network, hence different network decompositions can lead to decentralized conditions with different relative merits. }

\il{The main significance of Theorem~\ref{thm:Main_result_1} below is that it considers multiple such conditions associated with two different network decompositions and combines them together pointwise over frequency by means of appropriate homotopy arguments.}

\il{The significance of the conditions in   Theorem~\ref{thm:Main_result_1}
are discussed in remarks that follow.}

~

\begin{theorem}
Under Assumptions~\ref{assump:converter}-\ref{assump:stab_bus}, 	
the     equilibrium~\eqref{eq:equilibrium} of \kkl{the} power system~\eqref{eq:Line_dynamics}-\eqref{eq:bus_dynmics} with its  small-signal model~\eqref{eq:bipartite_eqt}   is locally asymptotically stable
if for  all~${\omega \in \overline{\R}_+}$ at least one  of the  following two statements is satisfied.			
\begin{itemize}[wide, labelwidth=!, labelindent=0.0cm]
\item {\bf Statement~1:} For every $j\in \mathcal{N}$, there exist   scalars	$\gamma_{j1}(\w)\geq 0$ and $\gamma_{j2}(\w)\geq 0$   and $\epsilon_{B_j}(\w)> 0$   		such that			
\begin{equation}
\begin{matrix}
B_j(\jw)\in \QC\left(\gamma_{j1}(\w)\Pi^{B_j}_1+\cdots~~~~~~~~~~~~ \right. \\\cdots+\left. \gamma_{j2}(\w)\Pi^{B_j}_2(\jw),\epsilon_{B_j}(\w) \right)		\end{matrix}
\label{eq:eqtStat1}
\end{equation}		
where  
\begin{equation}
\Pi^{B_j}_1=\begin{pmatrix}
0 &~& 1 \\ 1 &~& 0
\end{pmatrix}~~~~~~~~~~~~~~~~~~~~~~~~~~~~~~~~
\label{eq:pi1Bj}
\end{equation}
\begin{equation}
\Pi^{B_j}_2(\jw)=\begin{pmatrix}
\khl{-J_{B_j}(\jw)} & ~~0 \\ ~~0 & ~~\khl{J_{B_j}(\jw)^{-1}} 
\end{pmatrix}
\label{eq:pi2Bj}
\end{equation}	
with \khl{\begin{equation}
J_{B_j}(\jw)= \left|  \sum_{k: k\in \mathcal{E}_j} L_k(\jw)\right|  +  \sum_{k: k\in \mathcal{E}_j} \left|  L_k(\jw)\right|.
\label{eq:JBj}
\end{equation}	}
\item {\bf Statement~2:} For every $j\in \mathcal{N}$, there exist   scalars	$\delta_{j1}(\w)\geq 0$,  $\delta_{j2}(\w)\geq 0$,  $\delta_{j3}(\w)\geq 0$  and $\epsilon_{G_j}(\w)> 0$ and scalars~$\Pi^k_{11}(\jw)=(\Pi^k_{11}(\jw))^*\leq 0$, $\Pi^k_{12}(\jw)$ and $ \Pi^k_{22}(\jw) =(\Pi^k_{22}(\jw))^*\geq 0$, with $k=1,\cdots,n_\ell$,     satisfying {$ -\Pi^k_{12}(\jw)-\Pi^k_{12}(\jw)^*+ 2~\Pi^k_{22}(\jw)   \leq~0,$}  such  that	
\begin{equation}
\begin{matrix}
G_j(\jw)\in \QC\left(\delta_{j1}(\w)\Pi^{G_j}_1+\delta_{j2}(\w)\Pi^{G_j}_2+\cdots ~~~~\right. \\\cdots+\left.  \delta_{j3}(\w) \Pi^{G_j}_3(\jw),\epsilon_{B_j}(\w) \right)		\end{matrix}
\label{eq:eqtStat2}
\end{equation}
where  $G_j(\jw)= (\oplus_{k=1}^{n_\ell} L_k(\jw)) ~(a^r_j)^{\top}~ B_j(\jw)~a^r_{j}$ and  
\begin{equation}
\Pi^{G_j}_1=\begin{pmatrix}
0&  I_{n_\ell}\\ I_{n_\ell}&0
\end{pmatrix}~~~~~~~~~~~~~~~~~~~~~~~~~~~~~~~~~~~~
\label{eq:pi1Gj}
\end{equation}
\begin{equation}
\Pi^{G_j}_2=\begin{pmatrix}
-\khl{2I_{n_\ell}}  & ~~0 \\ ~~0 & ~~\khl{2^{-1}I_{n_\ell}}
\end{pmatrix}~~~~~~~~~~~~~~~~~~~~~~~~~~~~~~
\label{eq:pi2Gj}
\end{equation}	
\begin{equation}
\Pi^{G_j}_3(\jw)=\begin{pmatrix}
\oplus_{k=1}^{n_\ell} \Pi^k_{11}(\jw)~&  \oplus_{k=1}^{n_\ell} \Pi^k_{12}(\jw)\\   \oplus_{k=1}^{n_\ell} \Pi^k_{12}(\jw)^*  &  \oplus_{k=1}^{n_\ell} \Pi^k_{22}(\jw)~
\end{pmatrix}.\\[3ex]
\label{eq:pi3Gj}
\end{equation}
\end{itemize}
\label{thm:Main_result_1}
\end{theorem}

\begin{proof} 	See Appendix~\ref{app:Thm1_proof}.	\\	\end{proof}

\begin{remark}[Microgrid decomposition]
The~stability conditions in Statement~1 and Statement~2 are decentralized conditions that  depend on local bus/line dynamics.  They are obtained using two different decompositions of the microgrid that lead to appropriate graph separation arguments \cite{Les:11}. In particular, Statement~1 is derived {by means of}   the conventional microgrid decomposition into buses and lines.  Statement~2 on the other hand  is based on a different decomposition of the microgrid, {analogous to the one used in}~{\cite{LWL:20,Les:11}}, that leads to subsystems $G_j$ involving each bus $B_j$ and the lines~$L_k$ connected to it  {as follows from}  the sparsity structure of~$a^r_{j}$.\\
\end{remark}

\begin{remark}[Passivity and small-gain conditions]
\label{rem:link_with_passivity_and_small_gain}
The different conditions of Statement~1 and Statement~2  can be related to the usual passivity and small-gain conditions. In particular, in Statement~1, having  $\gamma_{j2}(\w)=0$ in~\eqref{eq:eqtStat1}  allows to recover the usual bus $B_j$  passivity conditions while choosing $\gamma_{j1}(\w)=0$ allows to recover a small-gain condition on each bus $B_j$ scaled \ilf{by $J_{B_j}$, where the latter depends on the neighbouring line dynamics $L_k$}.  For Statement~2,
we can have similar interpretations as earlier but   on  systems $G_j$ this time. For instance, 	   choosing  $\delta_{j2}(\w)=0$ and $\delta_{j3}(\w)=0$ in~\eqref{eq:eqtStat2}  allows to have a passivity   condition on $G_j$ while having  $\delta_{j1}(\w)=0$ and $\delta_{j3}(\w)=0$  allows     to recover a small-gain condition on $G_j$.    Finally, it is worth mentioning that in contrast to the passivity \ilff{condition in Statement 1,}
the third component in~\eqref{eq:eqtStat2} obtained with $\delta_{j1}(\w)=0$ and $\delta_{j2}(\w)=0$  allows  to take into account  the dynamics of the  lines   as  each  $\Pi^k$ can be   associated 	to each~$L_k$.\\ 
\end{remark}

\begin{remark}[Conservatism]
\label{rem:conservatisme}
Condition~\eqref{eq:eqtStat1} allows to combine passivity with small-gain conditions (see Remark~\ref{rem:link_with_passivity_and_small_gain}),  {thus} reducing the conservatism of more {conventional 
passivity based results often used in the literature}.
{Conditions~\eqref{eq:eqtStat1} and~\eqref{eq:eqtStat2} allow}
to reduce the conservatism by combining passivity with small-gain conditions but also with  other  {conditions} able to take into account {the} line dynamics {(see Remark~\ref{rem:link_with_passivity_and_small_gain}).}	
Note that when considered individually, neither Statement~1 nor Statement~2 is less conservative compared to {the} other. For instance,  Statement~1 {considers} the {more commonly used} bus/line decomposition  \ilff{and \eqref{eq:pi1Bj}} and does not take into account the 'strength' of coupling among the bus dynamics at each frequency. On the other hand, even {though} Statement~2 {allows} to consider this {coupling}, it may not always hold when the couplings {are too strong}.
Hence each statement has its own merits and a main  contribution  of Theorem~\ref{thm:Main_result_1} is to show that conditions~\eqref{eq:eqtStat1} and~\eqref{eq:eqtStat2} can be combined together pointwise over frequency by  an appropriate {homotopy argument (see {proof} in  Appendix~\ref{app:Thm1_proof})}, thus reducing the conservatism associated with these decentralized conditions. \\
\end{remark}

\begin{remark}[Control design]
\label{rem:control_design}
{The stability conditions stated in   Theorem~\ref{thm:Main_result_1}}	
can be used as   design protocols for the microgrid that {need} to be decided a priori, \ie  local
design rules at each bus which if satisfied   ensures stability of a general network.
{An approach when choosing such rules is to consider different conditions in different frequency ranges.}	For instance, the passivity  conditions can be used in regimes of higher gains as passivity 	 {holds} for {arbitrarily} large gains while small-gain conditions, {or conditions that take into account the strength of the coupling,} can  be considered in regimes with weaker {coupling} 	and potential phase lags.	{It should be noted that,}
the generalized KYP Lemma~\cite{iwasaki2005generalized} can  be used {to verify} 	if the different required properties are {satisfied}
in the  {corresponding}
frequency ranges.		
\end{remark}

\il{In Theorem~\ref{thm:Main_result_1}, the transfer \kkl{functions} $B_j(s), L_k(s)$ are in~$\RHi$. We consider here also the case \kkl{where} $L_k(s)$ has a pole on the imaginary axis at the origin, which is a more involved problem. This corresponds, for example, to the case where the lines are purely inductive, which is an assumption often made in AC grids at the transmission \ilc{level\footnote{\ilc{It should be noted that the swing equation with higher order generation dynamics has a small-signal model analogous to that in \eqref{eq:bipartite_eqtNN}.
}}.} For DC grids this assumption is less common, but it can be relevant in future  superconducting DC systems where the transmission lines have very small resistance (see e.g.~\cite{JLADSCA:93}).}

\kkl{ 
\ilc{We state below the class of} transfer functions $L_k(s)$ that we consider.}

\kkl{ \begin{assumption}
\label{assump:line_general_dynamic} For each $k$, we assume that  $L_k(s)$ is a proper positive real, proper real    rational transfer function with  a  stabilizable and detectable state-space realization  and with a pole at the origin and no other poles in $\overline{\C}_{+}$. 
\end{assumption} }

\kkl{ Similarly to~\eqref{eq:bipartite_eqt},	the  small-signal model of the microgrid can be represented as follows
\begin{equation}\left\lbrace
\begin{matrix}
\overline{V_B}(s)= -B(s)~\mathcal{A}~~ \overline{I_L}(s)+\overline{W_V}(s)\\
\overline{I_L}(s)= ~~L(s)~ \mathcal{A}^{\top}~ \overline{V_B}(s) +\overline{W_L}(s)\\
\end{matrix} \right.
\label{eq:bipartite_eqtNN}
\end{equation}		
where $ B(s)=\oplus_{j=1}^{n_b}  B_{j}(s)$ and $L(s)=\oplus_{k=1}^{n_\ell}  L_k(s)$
with $B_j(s) \in \RHi$  and  $L_k(s) \notin \RHi$.	
The signal    $\overline{W_V}(s)$  is the same as in~\eqref{eq:bipartite_eqt} while $\overline{W_L}(s)=\Theta_L(s) \overline{x_L}(0)$  where  $\overline{x_L}(0)$  is the initial \ilc{value} %condition
of~$( \overline{x_{L_k}}(t))_{k \in \mathcal{E}}$ (with $\overline{x_{L_k}}(t)$   the state vector of $L_k$) and
\begin{equation}
\Theta_L(s)= \left(\oplus_{k=1}^{n_\ell} C_{L_k} \right) \left(sI-\oplus_{k=1}^{n_\ell} A_{L_k} \right)^{-1}
\label{eq:Theta_L}
\end{equation}
with $C_{L_k}$ and $A_{L_k}$    the output matrix and  the state matrix in the state-space representation of $L_k(s)$. Note that $\Theta_L(s)$ has one pole at the origin (can be \ilff{deduced} from Assumption~\ref{assump:line_general_dynamic}), hence   $\overline{W_L}(s)$ is not in $\mathscr{H}_2$.}

\il{Theorem~\ref{thm:Main_result_2}  states that the conditions in  Theorem~\ref{thm:Main_result_1} can still be used to deduce convergence to an equilibrium point under an additional positivity condition at $\omega=0$.}

\il{\begin{theorem}			
Consider    the small signal model~\eqref{eq:bipartite_eqtNN} under Assumptions~\ref{assump:converter}-\ref{assump:line_general_dynamic}. Then, for all initial conditions, the voltage and the current deviations    $\overline{v_B}(t)$ and  $\overline{i_L}(t)$  converge to a \ilc{constant %equilibrium
value} if  	
\begin{alphalist}
\item 	For  all~${\omega \in \overline{\R}_+\setminus\{0\}}$, at least one of the statements of Theorem~\ref{thm:Main_result_1} is satisfied;\\
\item   For all $j\in \mathcal{N}$	
\begin{equation}
B_j(\j0)\in \QC\left(\Pi^{B_j}_1,\epsilon_{B_j}(0) \right).
\label{eq:main_2_eqt_2}
\end{equation}
for some $\epsilon_{B_j}(0)>0$.			
\end{alphalist}
\label{thm:Main_result_2}
\end{theorem}}

~	\begin{proof} 	See Appendix~\ref{app:Thm2_proof}.	\\	\end{proof}

\begin{remark}[Passivity at low frequencies]\label{rem:passivity_lowFreq}
\il{The passivity condition~\eqref{eq:main_2_eqt_2} \il{reduces to} %boils down
requiring that \mbox{$B_j(\j0)>0$}. This arises from the infinite gain of $L_k(\jw)$ at $\omega=0$. It should  be noted that this condition is also necessary for many classes of  \ilff{systems;} e.g. consider the \il{simple example of the}  negative feedback interconnection of a stable first order system $\dfrac{k_1}{s+\lambda}$ with  $\lambda\kkl{>}0$, \iclf{and}  an integrator $\dfrac{1}{s}$. In this case  \eqref{eq:main_2_eqt_2} reduces to $k_1\geq0$ and it is easy to see that the feedback interconnection is unstable if \ilff{$k_1<0$}.}
\end{remark}

\section{Examples}\label{sec:num_exp}		
To  demonstrate the applicability  of our results, we consider  two  generic examples  of DC~microgrids where each bus contains a controlled DC-DC converter connected to loads, see   for instance~\cite{BNDL:17,CRMFM:18}.

{	In particular, we will  consider configurations with parameters that have been chosen in the literature in a centralized way  to have good performance. We will then investigate whether stability can also be verified via the decentralized conditions derived in the paper. The significance of the latter is that they can be used as a decentralized design protocol for the network; \ie the choice of frequency ranges where the different statements in Theorem~\ref{thm:Main_result_1} are applied and the corresponding multipliers ($\Pi_i^{B_j}$ and $\Pi_i^{G_j}$), can be used as  local design rules for the converters through which stability of the network can be guaranteed.}

In our analysis we will start {with} Statement~1, \ie 
the bus/line decomposition, and consider first the more commonly used bus 
passivity condition ${ 	B_j(\jw)\in \QC(\Pi^{B_j}_1,\epsilon_{B_j}(\w))}.$  If this condition is not satisfied for all/some frequencies, we test the small-gain condition ${ 	B_j(\jw)\in \QC(\Pi^{B_j}_2(\jw),\epsilon_{B_j}(\w))}$ at those frequencies. If the previous condition is also  not satisfied at all/some of those frequencies, we {make use of} %invoke
Statement~2    and we test if ${ 	G_j(\jw)\in \QC(\Pi^{G_j}_1,\epsilon_{G_j}(\w))}$.  If the previous condition is still not satisfied at those frequencies,  we test  ${ 	G_j(\jw)\in \QC(\Pi^{G_j}_2,\epsilon_{G_j}(\w))}$ and we continue,  {if necessary, by testing whether ${ G_j(\jw)\in \QC(\Pi^{G_j}_3(\jw),\epsilon_{G_j}(\w))}$ can be satisfied for some choice of multipliers $\Pi^{G_j}_3(\jw)$.}

The examples we are considering deal with    two   common   situations.
\begin{itemize}
\item Microgrid with resistive loads
{where adjusting the converter control parameters to passivate the bus dynamics  can result  in a significant} voltage deviation from the nominal value.
\item Microgrid with ZIP loads dominated by their  constant power components which {can pose stability challenges.} 
\end{itemize}

Note that the numerical values used in these two examples are taken from~\cite{LSGVHW:15,ZSAM:16}.

\begin{figure}[t]
\centering
\includegraphics[width=0.6\linewidth]{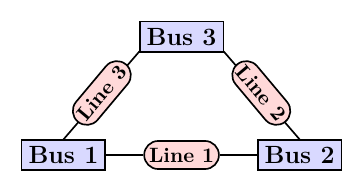}
\caption{Simplified representation of the microgrid considered.}
\label{fig:exp3_diagram}
\end{figure}

\klRev{The microgrid under consideration is composed of three buses as illustrated
in Fig.~\ref{fig:exp3_diagram}, where each bus is composed of a controlled buck converter connected
to a load. Under the common assumptions that \icr{an average model for the DC-DC converter dynamics can be used}
%		may be averaged
(resulting from  Assumption~\ref{assump:converter})
and that switching \icr{losses} may be ignored, the dynamics of each bus are given by}	
\begin{equation}
\left\lbrace
\begin{matrix}
&\diff{i_{j}(t)}{t} =\dfrac{1}{L_j}\left(-v_{B_j}(t)-R_{j} i_{j}(t) + u_j(t) \right)\\
&\diff{v_{B_j}(t)}{t}=\dfrac{1}{C_{j}}\left(i_{j}(t)- i_{\text{Load}_j}(t) +i_{B_j}(t)\right)\\
\end{matrix}\right. 
\label{eq:control_bubk_model}
\end{equation}			
where 	$i_{j}(t)$ is the bus internal current (the  inductor current),   $i_{B_j}(t)$ is the bus injection current from the network,  $v_{B_j}(t)$ is the bus output voltage and $u_j(t)$ is the control input of the buck converter.  $R_{j}$, $L_j$ and  $C_j$    are  bus filter resistance, inductance and capacitance  respectively. 
\klRev{ The control input, $u_{j}(t)$, is determined by the local controller at the bus and is actuated via the duty ratio of the DC-DC converter,
usually with the goal of regulating the output voltage and/or achieving load
sharing between the various DC-DC converters in the network. In this example,
it is given as the output of    a double~PI controller with the following model, {where} $x_{K_j}(t)$ {is} the controller state vector. }
\begin{equation}
\label{eq:control_struct}
\left\lbrace
\begin{matrix}
&\diff{x_{K_j}(t)}{t}=\hspace{-0.05cm}
\begin{pmatrix} 0 &0\\ K_{I_{v_{j}}}  &0
\end{pmatrix}x_{K_j}(t)+\dots\\& \hspace{1.5cm} \dots+\begin{pmatrix}1&0\\K_{P_{v_{j}}} &-1\end{pmatrix}\hspace{-0.15cm} \begin{pmatrix}e_{v_j}(t)\\i_{j}(t)\end{pmatrix} \\
&u_{j}(t)=  \begin{pmatrix} K_{P_{i_{j}}}\hspace{-0.15cm}K_{I_{v_{j}}} &~~ K_{I_{i_{j}}}  \end{pmatrix}x_{K_j}(t) +\cdots\\&~~~~~~~~\cdots+\begin{pmatrix} K_{P_{i_{j}}}\hspace{-0.15cm}K_{P_{v_{j}}} & -K_{P_{i_{j}}}  \end{pmatrix} \begin{pmatrix}e_{v_j}(t)\\i_{j}(t)\end{pmatrix}
\end{matrix}\right.
\end{equation}
{The} gains $K_{P_{v_{j}}}$ and  $K_{I_{v_{j}}}$ are the {proportional   and the
integral 
gains, respectively, } 
of the voltage PI controller~(outer controller) while  $K_{P_{i_{j}}} $ and $K_{I_{i_{j}}} $ are those of the current PI controller~(inner controller). The signal
$e_{v_j}(t)$ is the voltage tracking error given by $e_{v_j}(t)= v_{r_j}(t) - v_{B_j}(t)$ where $v_{r_j}(t)$ is  the desired bus output voltage which is adjusted according to  $i_{B_j}(t)$  such that  $v_{r_j}(t)=v_{r_\text{nom}}+R_{\text{droop}_j}i_{B_j}(t)  $
with $v_{r_\text{nom}}$   the nominal voltage and $R_{\text{droop}_j}$  the droop coefficient.

\klRev{Note that the state vector $x_{B_j}(t)$  in~\eqref{eq:bus_dynmics}  is in this case  $x_{B_j}(t)=\begin{pmatrix}
i_j(t) & v_{B_j}(t)& x_{K_j}^{\top}(t)
\end{pmatrix}^{\top}$ while  $f_{B_j}(x_{B_j}(t),i_{B_j}(t))$  is obtained by considering the right-hand side of  the differential equations in~\eqref{eq:control_bubk_model}-\eqref{eq:control_struct} after replacing $u_j(t)$, $e_{v_j}(t)$, $v_{r_j}(t)$ and $ i_{\text{Load}_j}(t)$ with their expressions given above.}		
Note that the current $ i_{\text{Load}_j}(t)$ in~\eqref{eq:control_bubk_model} is the load current and its expression depends on the load type as it will be shown in the two cases below.

\subsection{Microgrid with  resistive loads }\label{sec:expRes}
In this case,  each bus is connected to {a resistive} load  $R_{\text{Load}_j}>0$ and  $i_{\text{Load}_j}(t)$ is given by  $$i_{\text{Load}_j}(t)=R^{-1}_{\text{Load}_j}v_{B_j}(t). $$		
To investigate the stability of this microgrid, we start with the usual passivity argument  of Statement~1  corresponding to ${ 	B_j(\jw)\in \QC(\Pi^{B_j}_1,\epsilon_{B_j}(\w))}$. The analysis  reveals  that the different buses  are not passive over {all  frequencies} especially   {at high frequencies} as it can be seen in Fig.~\ref{fig:Exp1_pass}.

A common approach to enhance  bus passivity is to {increase    the droop coefficients~$R_{\text{droop}_j}$ until we  have}  ${ 	B_j(\jw)\in \QC(\Pi^{B_j}_1,\epsilon_{B_j}(\w))}$   at all   frequencies.  However, the consequence  will be an important output voltage deviation from the nominal  value $v_{r_\text{nom}}$ which makes this approach non practical for the {operation of the microgrid.}

To go beyond the passivity   \ilff{condition}   using Statement~1,   we investigate if   the small-gain condition ${ 	B_j(\jw)\in \QC(\Pi^{B_j}_2(\jw),\epsilon_{B_j}(\w))}$  is satisfied at least {at high frequencies}. The results  presented in  Fig.~\ref{fig:Exp1_sg} affirm that the %previous 
{latter} 
condition  is satisfied {at high frequencies where passivity failed}. Hence, the microgrid of Fig.~\ref{fig:exp3_diagram} with resistive loads is stable    since Statement~1 is always satisfied  at each frequency~${\omega \in \overline{\R}_+}$ as it can be seen in  Fig.~\ref{fig:Exp1_S1}.

{Therefore,} a scalable control protocol when considering such microgrids {in a general network topology can be to require the dynamics at each bus to satisfy} 
\begin{itemize}
\item {${ 	B_j(\jw)\in \QC(\Pi^{B_j}_1,\epsilon_{B_j}(\w))}$
at a prescribed low frequency range $\omega<\omega_c$.}
%at  low frequencies;
\item {${ 	B_j(\jw)\in \QC(\Pi^{B_j}_2(\jw),\epsilon_{B_j}(\w))}$  at  higher      frequencies $\omega\geq\omega_c$.}
\end{itemize}

\begin{figure*}[!t]
\centering \hspace{-0.0cm}
\begin{subfigure}{.625\columnwidth}
\includegraphics[width=1\columnwidth]{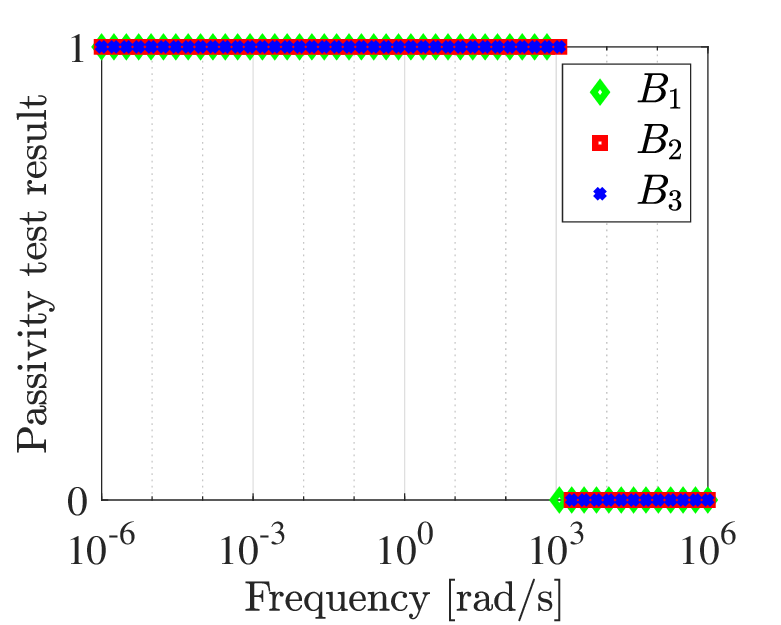}%
\caption{}%
\label{fig:Exp1_pass}%
\end{subfigure} \hspace{-0.0cm}
\begin{subfigure}{.625\columnwidth}
\includegraphics[width=1\columnwidth]{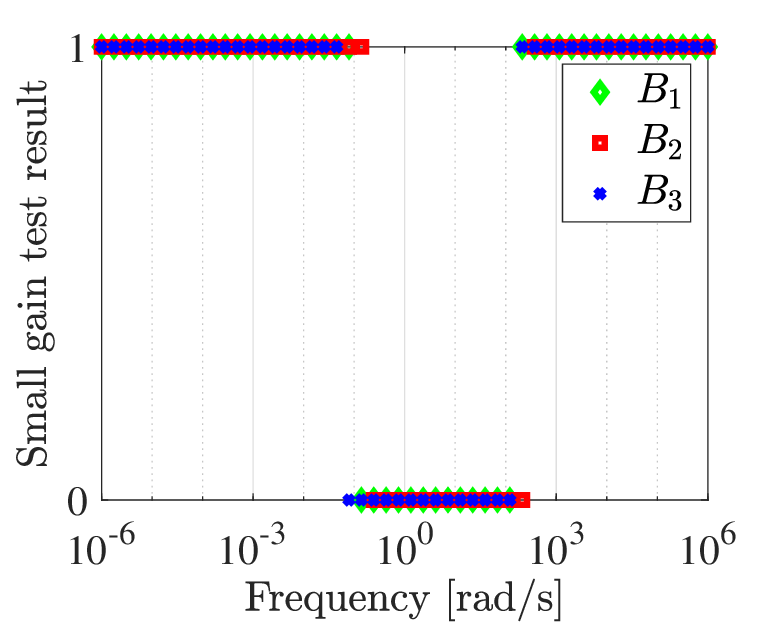}%
\caption{}%
\label{fig:Exp1_sg}%
\end{subfigure} \hspace{-0.0cm}
\begin{subfigure}{.625\columnwidth}
\includegraphics[width=1\columnwidth]{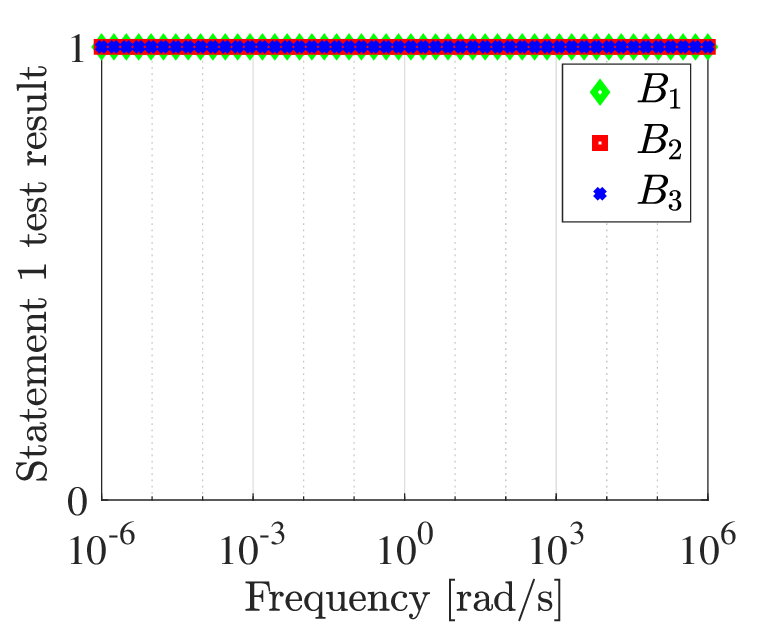}%
\caption{}%
\label{fig:Exp1_S1}%
\end{subfigure}%
\caption{Stability assessment  results of the microgrid of Fig.~\ref{fig:exp3_diagram} with resistive loads using: (a)~passivity, (b)~small gain and (c)~Statement~1 {(passivity and small gain)}. Note that at \ilc{each} %the considered
frequency,  a value at 0 on the vertical axis means that the  considered test has   failed while 1 means it has passed.}
\label{fig:E1_results}
\end{figure*}

\begin{figure*}[!b]
\centering \hspace{-0.0cm}
\begin{subfigure}{.652\columnwidth}
\includegraphics[width=1\columnwidth]{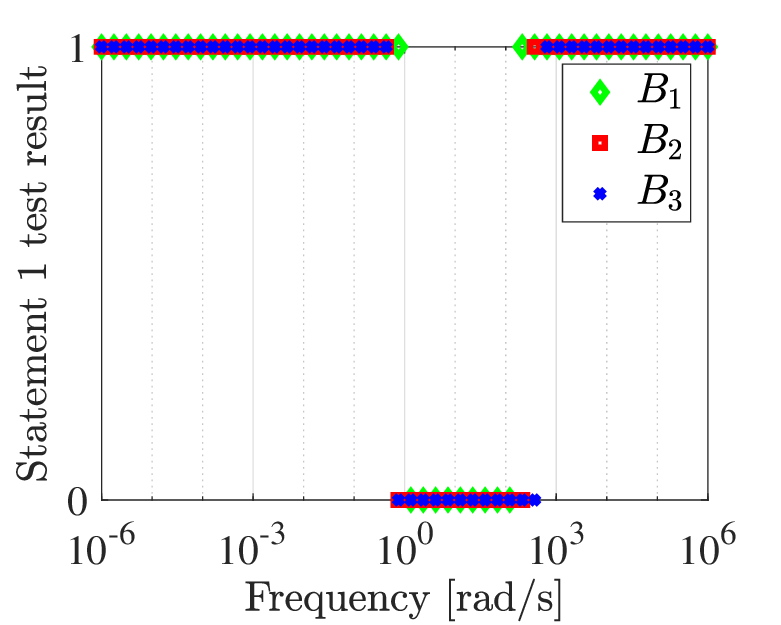}%
\caption{}%
\label{fig:Exp2_S1}%
\end{subfigure} \hspace{0cm}
\begin{subfigure}{.625\columnwidth}
\includegraphics[width=1\columnwidth]{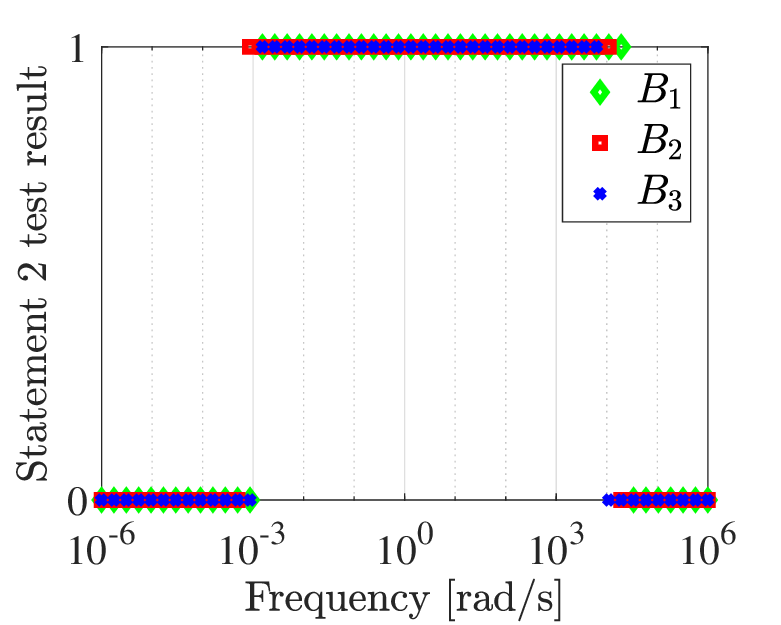}%
\caption{}%
\label{fig:Exp2_S2}%
\end{subfigure} \hspace{0cm}
\begin{subfigure}{.625\columnwidth}
\includegraphics[width=1\columnwidth]{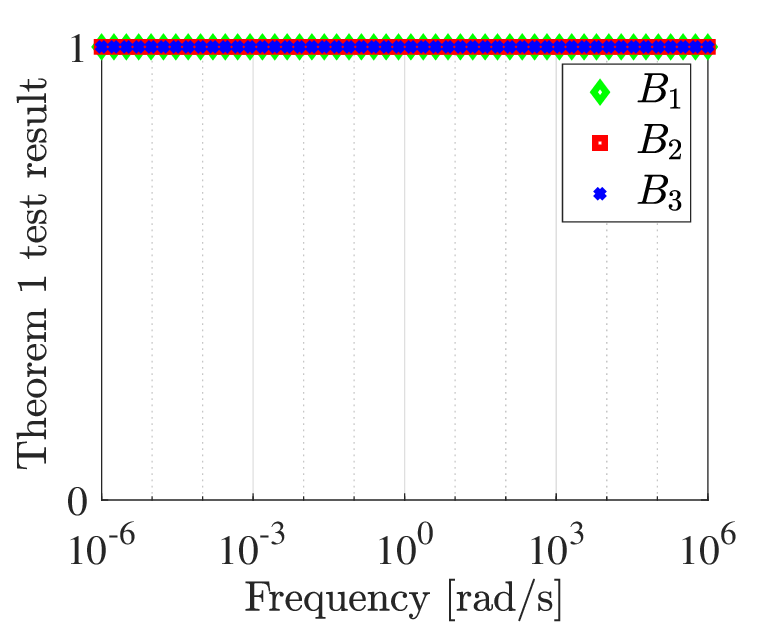}%
\caption{}%
\label{fig:Exp2_Thm1}%
\end{subfigure}%
\caption{Stability assessment  results of the microgrid of Fig.~\ref{fig:exp3_diagram} with ZIP loads where  $P_{\text{Load}_j}v_{r_j}^{-1}(t)> Z^{-1}_{\text{Load}_j}v_{r_j}(t)$  using: (a)~Statement~1, (b)~Statement~2 and (c)~Theorem~\ref{thm:Main_result_1}. Again note that  {at each} frequency, {a value} at 0 on the vertical axis means that the  considered test has   failed while 1 means it has passed.}
\label{fig:E2_results}
\end{figure*}

\subsection{Microgrid with  ZIP {loads}}\label{sec:expZIP}

We consider now {the case where} the loads are not just  resistive and they also  contain constant current  and constant power elements.
The load current $i_{\text{Load}_j}(t)$ becomes
$$ 	i_{\text{Load}_j}(t)=\overline{i}_{\text{Load}_j}+Z^{-1}_{\text{Load}_j}v_{B_j}(t)+P_{\text{Load}_j}v^{-1}_{B_j}(t) $$
where $\overline{i}_{\text{Load}_j}$,  $Z_{\text{Load}_j}=R_{\text{Load}_j}>0$ and $P_{\text{Load}_j}>0$ are \ilc{the current, impedance and power of the}  constant current load, the  constant impedance load and the constant power load  respectively.

The presence of constant power loads has a destabilizing effect on {the} bus dynamics since  they behave as negative {resistance} elements.
To guarantee bus stability using {a passivity} argument, the {effect of the constant  impedance loads is {larger  than} that of the  constant power loads~\cite{CLKKS:19,NSMMFT:20},} that is  $$ Z^{-1}_{\text{Load}_j}v_{r_j}(t)>P_{\text{Load}_j}v_{r_j}^{-1}(t). $$		
If the previous condition is satisfied, then it is possible to certify the stability of the microgrid with ZIP loads using Statement~1 in a similar way to  the case of resistive loads presented earlier.
Nevertheless, in many practical {cases,}  the constant power {loads}  can be larger than the constant impedance {loads}.  In this case,  Statement~1 alone will not be able to certify {microgrid} stability    as it can be seen in Fig.~\ref{fig:Exp2_S1}. {Note that in contrast to the case of resistive loads,
the passivity condition does {not hold} at low frequencies
while the small-gain condition holds.} Note also that none of these conditions {hold} in medium frequencies, see Fig.~\ref{fig:Exp2_S1}.

Statement~2 allows to go beyond   passivity and small-gain conditions of the conventional bus/line microgrid decomposition. In fact, the analysis reveals that by  considering the small-gain condition in  the new decomposition   ${ 	G_j(\jw)\in \QC(\Pi^{G_j}_2,\epsilon_{G_j}(\w))}$, Statement~2 is satisfied  at  medium frequencies where Statement~1 failed,  see Fig.~\ref{fig:Exp2_S2}.
Therefore,  we have     Statement~1 and/or  Statement~2 satisfied  at  each frequency~${\omega \in \overline{\R}_+}$ as it can be seen in   Fig.~\ref{fig:Exp2_Thm1},  and hence  the   microgrid of Fig.~\ref{fig:exp3_diagram} with ZIP loads {considered} (dominated by
%their P elements
{constant power loads}) is stable.

To conclude, when considering such microgrids {in a} general network topology, a scalable control protocol  {is to design} %will
%consist in
%design
controllers able to satisfy		
\begin{itemize}
\item ${ 	B_j(\jw)\in \QC(\Pi^{B_j}_2(\jw),\epsilon_{B_j}(\w))}$   {in a prescribed low frequency range}.
\item ${ 	G_j(\jw)\in \QC(\Pi^{G_j}_2,\epsilon_{G_j}(\w))}$  {in a prescribed medium frequency range}. %at  medium frequencies;
\item ${ 	B_j(\jw)\in \QC(\Pi^{B_j}_1,\epsilon_{B_j}(\w))}$  at  high frequencies.
\end{itemize}

\klRev{\begin{remark}[Post-analysis verification]
The \ir{analysis presented in} Sections~\ref{sec:expRes} and~\ref{sec:expZIP}
%are only valid
\ir{has been carried out for a} predefined set of individual frequencies.
\irr{This analysis provides the frequency ranges and the corresponding conditions, which define the protocol that provides stability guarantees. Once this step is completed,} 
we have used the generalized KYP lemma~\cite{iwasaki2005generalized} as a post-analysis verification tool to effectively verify that  the  results obtained  hold over the \ir{whole} prescribed frequency ranges.
\end{remark}}

\section{Conclusions}\label{sec:concl}
We have derived decentralized stability conditions for
DC~microgrids. Our analysis takes into account {the line}
dynamics and {also allows higher order models} %tolerates complex models
for the DC-DC converters at each bus. By exploiting various decompositions
of the network, we have derived {multiple} decentralized input-output stability
conditions    that {also} allow to exploit the coupling
of each bus with neighboring lines.  We have used appropriate 
homotopy arguments   to combine these
conditions {pointwise over} frequency thus reducing the conservatism {in the} analysis. The applicability of the obtained results has been illustrated through  {examples.}

\appendix

\section{Proof of Theorem~\ref{thm:Main_result_1}}	\label{app:Thm1_proof}

The following lemma is {used in the proof of} %required to prove
Theorem~\ref{thm:Main_result_1}.
It gives  sufficient conditions {that} ensure that   the  point -1 is not  included in the   {eigenloci of the} return-ratio of   {the negative feedback interconnection of two  stable linear systems.}

\begin{lemma} Consider a negative feedback interconnection of   $S_1\in \RHi^{m\times n}$ and $S_2 \in \RHi^{n \times m}$. {The} \mbox{point -1 }  is not included in the eigenloci of the  return-ratio  $S_2(s)S_1(s)$ when evaluated on the imaginary axis $\jR$, that is  $-1 \notin \lambda_i\left(  S_1(\jw)S_2(\jw)\right)$ with $\omega \in \overline{\R} $,  if
there {exists} a scalar~$\epsilon^\Pi(w)>0$ and a  matrix  $\Pi=\begin{pmatrix}
\Pi_{11} & \Pi_{12}\\ \Pi_{12}^* & \Pi_{22}
\end{pmatrix}$, where $\Pi_{11}:\j\overline{\R}  \rightarrow \C^{m \times m }$, ${\Pi_{12}}:\j\overline{\R}  \rightarrow \C^{m \times n }$ and $\Pi_{22}:\j\overline{\R}  \rightarrow \C^{n \times n }$   such that for every $\omega \in \overline{\R}_+$,  the  following  conditions hold	
\begin{equation}
S_1(\jw)\in \QC\left(\Pi(\jw),\epsilon^\Pi(\w) \right) 
\label{eq:Dis_QC_stab_C1}
\end{equation}		
and
\begin{equation}
S_2(\jw)\in \QCinv\left( \Pi(\jw),0  \right). 
\label{eq:Dis_QC_stab_C2}
\end{equation}	
% with $\Pi_{11}(\jw)\leq 0$ and  $\Pi_{22}(\jw)\geq 0$
\label{lem:graph_sep_lemma}
\end{lemma}

\begin{proof}
Suppose that   $-1 \in \lambda_i\left(  S_1(\jw)S_2(\jw)\right)$ which means that    $\left(I+  S_1(\jw)S_2(\jw)\right)$ is not  invertible,    then  there exists a \ir{non zero}  vector $w(\jw)  \in\mathbb{C}^{m}$ different from zero  such that
$(I+   S_1(\jw)S_2(\jw))w(\jw)=0.$   Letting~$z(\jw)=-S_2(\jw)w(\jw)$, the previous equality becomes    $ w(\jw)- S_1(\jw) z(\jw)=0$ and we    obtain  $ w(\jw)= S_1(\jw) z(\jw)$ together with $z(\jw)=-S_2(\jw)w(\jw)$.		
Let $\sigma(\jw)=\begin{pmatrix}       w(\jw)^* ~~ z(\jw)^*  \end{pmatrix}^*$. 	After pre and post multiplying  the expanded forms (expanded as in~\eqref{def:QC}) of conditions~\eqref{eq:Dis_QC_stab_C1}  and~\eqref{eq:Dis_QC_stab_C2}   by $z(\jw) $ and $w(\jw) $     respectively, we~obtain $\sigma(\jw)^* \Pi(\jw) \sigma(\jw)
\geq  \kl{\epsilon^\Pi(\w)}~w(\jw)^*w(\jw)$ and  $\sigma(\jw)^* \Pi(\jw) \sigma(\jw)
\leq  0$
which is a contradiction. Therefore, if conditions  \eqref{eq:Dis_QC_stab_C1} and~\eqref{eq:Dis_QC_stab_C2} are satisfied  then  $-1 \notin \lambda_i\left(  S_1(\jw)S_2(\jw)\right)$ for all
$\omega \in \overline{\R}_+$. \ilc{Due to the symmetry of $\lambda_i(S_1(\jw)S_2(\jw)))$ about the real \ilff{axis,} 
$-1 \notin \lambda_i\left(  S_1(\jw)S_2(\jw)\right)$
also holds for all $\omega \in \overline{\R}$.} $\qed$
\end{proof}

The local  asymptotic stability of the equilibrium~\eqref{eq:equilibrium} of the DC microgrid~\eqref{eq:Line_dynamics}-\eqref{eq:bus_dynmics}       is investigated via the internal stability \cite[Def 5.2]{ZDG:95} of interconnection~\eqref{eq:bipartite_eqt} \cite[Lem 5.3]{ZDG:95}, \cite[Thm 3.7]{Kha:92}.
In particular, it is sufficient to show that the closed-loop  transfer functions  of interconnection~\eqref{eq:bipartite_eqt} (as defined in \cite[Lem 5.3]{ZDG:95})  have no poles in the closed right~half-plane~$\overline{\C}_+$. \kkl{Therefore, as $B(s)$ and $L(s)$ have no poles in~$\overline{\C}_+$, we only have to show that   $(I+Q_O(s))^{-1}$ has no poles in~$\overline{\C}_+$ with $Q_O(s)$ being the return-ratio of interconnection~\eqref{eq:bipartite_eqt}   given by 
\begin{equation}
Q_{{O}}(s)=B(s) \mathcal{A} L(s) \mathcal{A}^{\top}\label{eq:OCRR}.
\end{equation}}	Using ideas from~\cite{Les:11,LWL:20},  interconnection~\eqref{eq:bipartite_eqt} can \ilc{also} be represented equivalently as an interconnection of two systems  $G(s)$ and   $A$  such that 
\begin{equation}
\left\lbrace 	\begin{matrix}
& Y(s)&=& ~G(s) ~X(s)\\
& X(s)&=& -A~~ ~Y(s) 
\end{matrix}\right. 
\label{eq:GA}
\end{equation}
where  $G(s)=\oplus_{j=1}^{n_b}G_j(s)$ with 	\begin{equation}
G_j(s)= L(s) ~(a^r_{j})^{\top}~ B_j(s)~a^r_{j}~\vspace{-0.3cm}	
\label{eq:G_j}
\end{equation}
and
\begin{equation} A= M M^*	\vspace{-0.1cm}	 \label{eq:GA_data}
\end{equation}
with $M=\big( M_1^*~  \dots ~ M_{n_b}^*\big)^*$ where  \vspace{-0.1cm}
\begin{equation}
M_j=\oplus_{k=1}^{n_\ell} \Big( r^j_k\Big)~~~~~	r^j_k=\left\lbrace \begin{matrix}
1~~~~ ~\text{if $\mathcal{A}_{j k}\neq0$ }\\
0~~~~ ~\text{otherwise}~~~
\end{matrix}\right.~~
\label{eq:r_ik}
\end{equation}
\kkl{We define  	the return-ratio of interconnection~\eqref{eq:GA} as
\begin{equation}
Q_N(s)=G(s)MM^*.\label{eq:NCRR}
\end{equation}  }
The proof  of Theorem~\ref{thm:Main_result_1} uses  the quadratic graph separation arguments of Lemma~\ref{lem:graph_sep_lemma}, analogous to an IQC analysis, %~\cite{MeRa:97},
on representations~\eqref{eq:bipartite_eqt} and \eqref{eq:GA},  together with a homotopy argument, with stability deduced  	using the multivariable Nyquist criterion~\cite{DeY:80}. \ilc{In particular, a main feature of the proof is that it allows to combine pointwise over frequency various decentralized conditions associated with the two different network decompositions\footnote{\ilc{It should be noted that \ilff{since} two different network decompositions are considered in the same stability condition, a classical dissipativity or IQC analysis are not directly applicable.}}~\eqref{eq:bipartite_eqt}, \eqref{eq:GA}. 
}\\[2ex]
The proof consists of four parts. In {\bf Part~1},  we show that if Statement~1  is true  for  \ilff{each} ${\omega \in \overline{\R}_+}$, then  the \mbox{point -1 }  is not included in the eigenloci of~\kkl{$Q_{{O}}(s)$ in \eqref{eq:OCRR}}.	
In {\bf Part~2}, we show  in the same way that if Statement~2   is true for  \ilff{each} ${\omega \in \overline{\R}_+}$, then  the \mbox{point -1 }  is not included in the eigenloci of~\kkl{$Q_{{N}}(s)$  in~\eqref{eq:NCRR}}. In {\bf Part~3}, we show that if  the \mbox{point -1 }  is not included in the eigenloci of~\eqref{eq:NCRR}    then it is not included in the  eigenloci $Q_O(s)$ of~\eqref{eq:OCRR}. Finally, we deduce in {\bf Part~4} using homotopy arguments  that if  at each frequency either
Statement~1 or Statement~2 is true, then  interconnection~\eqref{eq:bipartite_eqt} is  stable.\\[1ex]
\textbf{$\bullet$ Part~1:} In this part, we
demonstrate   that having  Statement~1 satisfied  at each $\w \in \overline{\R}_+$ is sufficient
for {conditions}~\eqref{eq:Dis_QC_stab_C1}	and \eqref{eq:Dis_QC_stab_C2} {in} Lemma~\ref{lem:graph_sep_lemma} to~hold. This is done by exploiting {the following} two properties {of interconnection}~\eqref{eq:bipartite_eqt}.

\begin{enumerate}
\item {\it Passivity of the {power lines}:} 	
From~\eqref{eq:Line_dynamics}, it can be deduced that the line dynamics are passive which can be expressed in the frequency domain at each frequency $\w \in \overline{\R}_+$ as 	$ L_k(\jw)\in \QCinv\left(\Pi^{B_j}_1,0\right) $ 	with $\Pi^{B_j}_1$    given by~\eqref{eq:pi1Bj}. This is equivalent to   	\begin{equation}
\mathcal{A} L(\jw)  \mathcal{A}^{\top}\in \QCinv\left(\Pi^{B}_1,0\right)  
\label{eq:ALA_QC_Pi1}
\end{equation}
with $\Pi^{B}_1=I_{n_b} \otimes \Pi^{B_j}_1$.\\

\item {\it Bounding the {power line}   induced $\infty$-norm:}	
\khl{	The induced  $\infty$-norm of $\Xi(\jw)=\mathcal{A} L(\jw)  \mathcal{A}^{\top}$ is given by
$\max_j \left( | \sum_{k: k\in \mathcal{E}_j} L_k(\jw)| +  \sum_{k: k\in \mathcal{E}_j} \left| L_k(\jw)\right|\right)$.		 To have a bound on this  norm, we \ilf{consider a scaling by the matrix} $J_B(\jw)^{-1}=\left(\oplus_{j=1}^{n_b}{J_{B_j}}(\jw)\right)^{-1}$, \ilf{where $J_{B_j}(\jw)$ is} given by \eqref{eq:JBj}. \ilf{In particular, $\left\| J_B(\jw)^{-1} \Xi(\jw) \right\|_\infty$  is} given by 		$$\max_j  J_{B_j}(\jw)^{-1} \Big( | \sum_{k: k\in \mathcal{E}_j} L_k(\jw)| +  \sum_{k: k\in \mathcal{E}_j} | L_k(\jw)|\Big)  $$ which is equal to~1.
\ilf{Furthermore, we also have} 
$\left \| J_B(\jw)^{-1} \Xi(\jw)^*  \right\|_\infty=1$.
Now, let $F(\jw)=J_B(\jw)^{-1} \Xi(\jw)^* J_B(\jw)^{-1} \Xi(\jw)$. The spectral radius of  $F(\jw)$ can be bounded using the \ilf{induced} $\infty$-norm as follows
\begin{align}
%		\begin{split}
\rho(F(\jw))&\leq  \left\|F(\jw) \right\|_\infty \nonumber\\
&\leq \left\|J_B(\jw)^{-1} \Xi(\jw)^*\right\|_\infty \left\|J_B(\jw)^{-1} \Xi(\jw)\right\|_\infty\nonumber\\
&\leq 1 \label{eq:rho_leq1}
%		\end{split}
\end{align}
\ilf{We now note that $F(\jw)$ has real eigenvalues as it is the product of a positive definite and a Hermitian matrix.}
Therefore, \ilf{since from \eqref{eq:rho_leq1}} $\rho(F(\jw))\leq 1$,  the eigenvalues of $J_B(\jw)^{-1} (\mathcal{A} L(\jw)  \mathcal{A}^{\top})^*J_B(\jw)^{-1} \mathcal{A} L(\jw)  \mathcal{A}^{\top}-I_{n_b}$ are less than or equal to zero.	
Moreover, as  $J_B(\jw)$ is positive definite,  the eigenvalues of $(\mathcal{A} L(\jw)  \mathcal{A}^{\top})^*J_B(\jw)^{-1} \mathcal{A} L(\jw)  \mathcal{A}^{\top}-J_B(\jw)$ are \ilf{also} less than or equal to zero\footnote{\ilf{In particular, we exploit the fact that $J_B(\jw)^{-1}P(\jw)-I_{n_b}$, where $P(\jw)=\Xi(\jw)^*J_B(\jw)^{-1}\Xi(\jw)\geq0$ has the same nonzero eigenvalues as $(J_B(\jw))^{-\frac{1}{2}}(P(\jw)-J_B(\jw))(J_B(\jw))^{-\frac{1}{2}}$. Also the latter being negative semidefinite implies  $P(\jw)-J_B(\jw)\leq0$.}} \ilf{leading  to}
$$   (\mathcal{A} L(\jw)  \mathcal{A}^{\top})^*J_B(\jw)^{-1} \mathcal{A} L(\jw)  \mathcal{A}^{\top} \leq J_B(\jw)$$ 	
which can be written   \ilff{in compact}  form as 	
\begin{equation}
\mathcal{A} L(\jw)  \mathcal{A}^{\top}\in \QCinv\left(\Pi^{B}_2(\jw),0\right)
\label{eq:ALA_QC_Pi2}
\end{equation}	with $
\Pi^{B}_2(\jw)=\begin{pmatrix}
-J_B(\jw)    & 0 \\0 & J_B(\jw)^{-1}
\end{pmatrix}.$ }	
\end{enumerate}
% Thereafter
\ilf{As} $\mathcal{A} L(\jw)  \mathcal{A}^{\top}$ satisfies \eqref{eq:ALA_QC_Pi1} and~\eqref{eq:ALA_QC_Pi2}, it will  satisfy  $$	\mathcal{A} L(\jw)  \mathcal{A}^{\top}\in \QCinv\left(\gamma_1(\w)\Pi^{B}_1+\gamma_2(\w)\Pi^{B}_2(\jw),0\right)$$ with $\gamma_1(\w)\ge0$ and $\gamma_2(\w)\ge0$. Moreover, to exploit the     diagonal structure of  $\Pi^{B}_1$  and $\Pi^{B}_2(\jw)$, the scalars  $\gamma_1(\w)$ and $\gamma_2(\w)$ can be replaced by two  diagonal matrices $\Gamma_1(\w)= I_2 \otimes \left(\oplus_{j=1}^{n_b}\gamma_{j1}(\w) \right)  $ and $\Gamma_2(\w)=I_2\otimes \left(\oplus_{j=1}^{n_b}\gamma_{j2}(\w)  \right)$, with  $\gamma_{j1}(\w)\geq 0$ and $\gamma_{j2}(\w)\geq 0$ \ilff{associated to}  $\Pi^{B_j}_1$ and $\Pi^{B_j}_2(\jw)$ \ilff{respectively}. We can thus~write   
\kkl{\begin{equation}
\hspace{-0.0cm}\mathcal{A} L(\jw)  \mathcal{A}^{\top} \hspace{-0.125cm}\in  \QCinv\left(\Gamma_1(\w)\Pi^{B}_1+\Gamma_2(\w)\Pi^{B}_2(\jw),0\right)
\label{eq:ALA_QC_Con}
\end{equation}}Therefore,
%to invoke Lemma~\ref{lem:graph_sep_lemma},}
%it is sufficient
for~\eqref{eq:Dis_QC_stab_C1}	and \eqref{eq:Dis_QC_stab_C2} of Lemma~\ref{lem:graph_sep_lemma} to hold it is sufficient at each $\w \in \overline{\R}_+$ to find  a   scalar~$\epsilon_B(\w)>0$   such that $$\kkl{	\oplus_{j=1}^{n_b}B_j(\jw)\in \QC\left(\Gamma_1(\w)\Pi^{B}_1+\Gamma_2(\w)\Pi^{B}_2(\jw),\epsilon_B(\w)\right)}.$$
Due to the  diagonal structure of  $\Gamma_1(\w)\Pi^{B}_1+\Gamma_2(\w)\Pi^{B}_2(\jw)$, the previous condition can be decomposed into  $n_b$ conditions given by~\eqref{eq:eqtStat1}
with  $\epsilon_{B_j}(\w)\geq \epsilon_B(\w){>0}$.\\
{Hence if Statement~1 holds for each $\w \in \overline{\R}_+$ then \eqref{eq:Dis_QC_stab_C1}	and \eqref{eq:Dis_QC_stab_C2} hold, and it therefore follows from
%Then, %using
Lemma~\ref{lem:graph_sep_lemma} that}
the \mbox{point -1 }  is not included in the eigenloci of the return-ratio  \kkl{$Q_O(s)$ of   interconnection~\eqref{eq:bipartite_eqt} given by~\eqref{eq:OCRR}}.\\[1ex] 
\textbf{$\bullet$ Part~2:} Similarly to {\bf Part~1}, we prove that
if   Statement~2 is satisfied  at each $\w \in \overline{\R}_+$ then it is sufficient
for~\eqref{eq:Dis_QC_stab_C1}	and \eqref{eq:Dis_QC_stab_C2} of Lemma~\ref{lem:graph_sep_lemma} to~hold. This is done by exploiting the {following properties and structure} of interconnection~\eqref{eq:GA}.	
\begin{enumerate}
\item {\it Positivity and symmetry of $A$:} 	
As $A=M M^*\geq 0$ and $A=A^*$, we have  $ -A^*-A\leq 0$ 	which rewrites in its compact form as   
\begin{equation}
A\in \QCinv\left(\Pi^{G}_1,0\right)  
\label{eq:A_QC_Pi1}
\end{equation}
where $\Pi^{G}_1 =I_{n_b} \otimes\Pi^{G_j}_1$ with $\Pi^{G_j}_1$ given  by~\eqref{eq:pi1Gj}. \\

\item {\it Bounding the  induced  $\infty$-norm of $A$:}	Using~\eqref{eq:GA_data} and \eqref{eq:r_ik}, the induced  $\infty$-norm of $A$  defined as the maximum row sum  $ r_k^j \sum_{j=1}^{n_b} r^j_k$ is always equal to 2 since  \mbox{$\sum_{j=1}^{n_b} r^j_k=2$} as the \ilff{graph associated with the microgrid}  is connected. \khl{Therefore,   $\left\| A\right\|_\infty=2$ \ilf{and since
$
\rho(A)\leq\left\| A\right\|_\infty$  with $A$ symmetric}, we obtain    $ A^* A\leq 4 I_{n_bn_\ell}$ which can be written  \ilff{in   compact} form as}
\begin{equation}
A\in \QCinv\left(\Pi^{G}_2,0\right) 
\label{eq:A_QC_Pi2}
\end{equation}
where $\Pi^{G}_2=I_{n_b}\otimes\Pi^{G_j}_2$ with $\Pi^{G_j}_2$ given by~\eqref{eq:pi2Gj}. \\

\item {\it Exploiting the structure of $A$:} To take into account {the fact that} the lines are part of the new subsystems $G_j$, we use  $\Pi^{G}_3(\jw)=I_{n_b}\otimes\Pi^{G_j}_3(\jw)$ with $\Pi^{G_j}_3(\jw)$ given {by~\eqref{eq:pi3Gj}}. With this particular choice of   $\Pi^{G}_3(\jw)$, we can  associate each  $\Pi^k(\jw)$   to  each  line $L_k(\jw)$. Note that  if  $-\Pi^k_{12}(\jw)-\Pi^k_{12}(\jw)^*+ 2~\Pi^k_{22}(\jw)   \leq 0$  together with $ \Pi^k_{11}(\jw) \leq 0$ then   
\begin{equation}
A\in \QCinv\left(\Pi^{G}_3(\jw),0\right) 
\label{eq:A_QC_Pi3}
\end{equation}  is always satisfied. {To see this, note that by} exploiting %\todoiny{previous phrase does not read well}
the sparsity    of $A$ and the structure of $\Pi^{G}_3(\jw)$,  the left hand  side of the expanded form  \ilff{of the condition in \eqref{eq:A_QC_Pi3}}    is given by 
\begin{equation*}
\begin{matrix}
I_{n_b}	\otimes ( \oplus_{k=1}^{n_\ell} \Pi^k_{11}(\jw)) +\cdots\\   \cdots+ MM^*(I_{n_b} \otimes ( \oplus_{k=1}^{n_\ell} \Pi^k_{22}(\jw)))MM^* -\cdots\\   \cdots -
M ( \oplus_{k=1}^{n_\ell} \Pi^k_{12}(\jw)) M^*\hspace{-0.1cm}-\hspace{-0.1cm}M ( \oplus_{k=1}^{n_\ell} \Pi^k_{12}(\jw)^*)M^*
\end{matrix}
\end{equation*}		which is equal to 
\begin{equation}
\begin{matrix}
&I_{n_b}	\otimes ( \oplus_{k=1}^{n_\ell} \Pi^k_{11}(\jw)) +\cdots \\& \hspace{-0.25cm}\cdots\hspace{-0.05cm}+\hspace{-0.05cm}
M \big( \oplus_{k=1}^{n_\ell}\big(\hspace{-0.05cm}-\hspace{-0.05cm} \Pi^k_{12}(\jw)\hspace{-0.05cm}-\hspace{-0.05cm} \Pi^k_{12}(\jw)^*\hspace{-0.05cm}+\hspace{-0.05cm}\cdots \\& ~~~~~~~~~~~~ \hspace{-0.5cm}\cdots+ 2 \Pi^k_{22}(\jw)\big)\big)M^*.
\end{matrix}
\label{eq:AforPi3}
\end{equation}	
Hence   $ \Pi^k_{11}(\jw) \leq0$  together with $-\Pi^k_{12}(\jw)-\Pi^k_{12}(\jw)^*+ 2 \Pi^k_{22}(\jw)    \leq 0$   are sufficient for~\eqref{eq:AforPi3} to be negative semi-definite and {hence} for condition~\eqref{eq:A_QC_Pi3}  to hold.
\end{enumerate}
Therefore,  as $A$ satisfies~\eqref{eq:A_QC_Pi1}, \eqref{eq:A_QC_Pi2} and \eqref{eq:A_QC_Pi3}   and using similar arguments to those {in} \textbf{Part~1},   we obtain  
\kkl{	\begin{equation}
\begin{matrix}
A\in  \QCinv	
\left(\Delta_{1}(\w)\Pi^{G}_1+\Delta_{2}(\w)\Pi^{G}_2+\cdots ~~~~\right. \\\cdots+\left.  \Delta_{3}(\w) \Pi^{G}_3(\jw),\epsilon_{G}(\w) \right)
\end{matrix}
\label{eq:A_QC_Con}
\end{equation}}where
$\Delta_i(\w)= I_{2n_\ell} \otimes  \left(\oplus_{j=1}^{n_b}\delta_{ji}(\w) \right)  $ with $i\in \{1,2,3\}$, {%where we have used the fact that
\ilc{and} $\delta_{ji}(\w)\geq0$.}	
Therefore,  {for~\eqref{eq:Dis_QC_stab_C1}	and \eqref{eq:Dis_QC_stab_C2} in Lemma~\ref{lem:graph_sep_lemma} to hold it is sufficient at} each $\w \in \overline{\R}_+$ to find  a   scalar~$\epsilon_G(\w)>0$   such that 
\kkl{	\begin{equation*}
\begin{matrix}
\oplus_{j=1}^{n_b}G_j(\jw)\in \QC	
\left(\Delta_{1}(\w)\Pi^{G}_1+\Delta_{2}(\w)\Pi^{G}_2+\cdots ~~~~\right. \\\cdots+\left.  \Delta_{3}(\w) \Pi^{G}_3(\jw),\epsilon_{G}(\w) \right).
\end{matrix}
\end{equation*}}
Due to the  diagonal structure of  $\Delta_i$ and $\Pi^{G}_i$, the previous condition can be decomposed into  $n_b$ conditions given by~\eqref{eq:eqtStat2}
with  $\epsilon_{G_j}(\w)\geq \epsilon_G(\w){>0}$. \\
%Then, using Lemma~\ref{lem:graph_sep_lemma}, the \mbox{point -1 }
{Hence if Statement~2 holds for each $\w \in \overline{\R}_+$ then \eqref{eq:Dis_QC_stab_C1}	and~ \eqref{eq:Dis_QC_stab_C2} hold, and it therefore follows from
Lemma~\ref{lem:graph_sep_lemma} that}
the \mbox{point -1 } is not included in the eigenloci of the return-ratio \kkl{$Q_N(s)$  of  interconnection~\eqref{eq:GA} given~by~\eqref{eq:NCRR}}.\\[1ex]
\textbf{$\bullet$ Part~3:} We {show} %reveal
in this part that if   the \mbox{point -1 }  is not included in the eigenloci  of \kkl{$Q_N(s)$ in~\eqref{eq:NCRR}}, then it is also not included in   the eigenloci of \kkl{$Q_O(s)$~in~\eqref{eq:OCRR}}.  %With simple operation we can show
A simple argument shows  that both return-ratios have the same non-zero eigenvalues. %This is done as follows.
In particular, the  matrix  $G(s) MM^* $  in \kl{$Q_N(s)$} has the same nonzero eigenvalues as
$ M^* G(s) M $ which  rewrites  as    $ \sum_{j=1}^{\ilc{n_b}} M_j  G_j(s) M_j^*$.
{Then} using the expressions of $G_j(s)$ and $M_j$ given by~\eqref{eq:G_j} and~\eqref{eq:r_ik} and the fact that  $L(s)M_j=M_j L(s)$, the previous summation becomes   $L(s)\sum_{j=1}^{n_b}   M_j(a^r_{j})^{\top} B_j(s) a^r_{j}  M_j^*$  which rewrites as
$ L(s) \mathcal{A}^{\top} B(s)\mathcal{A}$. Finally, note that   $ L(s) \mathcal{A}^{\top} B(s)\mathcal{A}$   has the same nonzero eigenvalues as
$ B(s) \mathcal{A}  L(s)  \mathcal{A}^{\top} $ which is the return-ratio \kkl{$Q_O(s)$}.  Hence, 	if the \mbox{point -1 }  is not included in the eigenloci of the return-ratio~\kkl{$Q_N(s)$}  of interconnection~\eqref{eq:GA}, then it is also not included in the eigenloci of the return-ratio~\kkl{$Q_O(s)$}  of     interconnection~\eqref{eq:bipartite_eqt}.\\[1ex]
\textbf{$\bullet$ Part~4:}
{We show now using a homotopy argument that when at each frequency Statement~1 or Statement~2 is satisfied then the point $-1$ is also not encircled by the eigenloci of the return-ratio of   \eqref{eq:bipartite_eqt} or~\eqref{eq:GA}, and hence stability can be deduced.}
We define  the following linear  homotopy for   interconnection~\eqref{eq:bipartite_eqt}: $\mathcal{A} L_\tau(\jw) \mathcal{A}^{\top}=\mathcal{A} L(\jw) \mathcal{A}^{\top}$ and $B_{j,\tau}(\jw)=\tau B_{j}(\jw)$ with $\tau\in[0,1]$. Similarly, we define \ilc{an analogous} homotopy for the  interconnection~\eqref{eq:GA} as $A_\tau=A$ and $G_{j,\tau}(\jw)=\tau {G_j}(\jw)$ with the same $\tau\in[0,1]$. We will show that \ilc{throughout these} homotopies,  {condition}~\eqref{eq:eqtStat1} together with~\eqref{eq:ALA_QC_Con} and  {condition}~\eqref{eq:eqtStat2} together with~\eqref{eq:A_QC_Con}  remain satisfied, \ilc{i.e. when}   $B_j(\jw)$, $\mathcal{A} L(\jw) \mathcal{A}^{\top} $, $B_j(\jw)$ and  $A$ and  are replaced by $B_{j,\tau}(\jw)$, $\mathcal{A} L_\tau(\jw) \mathcal{A}^{\top}$,  $G_{j,\tau}(\jw)$ and $A_\tau$ respectively. \\
Conditions~\eqref{eq:ALA_QC_Con} and \eqref{eq:A_QC_Con} are always satisfied \ilc{when $\gamma_{ji}(\w)\geq 0$, $\delta_{ji}(\w)\geq 0$,  $J_{B_j}(\jw)$ is given by~\eqref{eq:JBj}, $ \Pi^k_{11}(\jw) \leq0$, and} $-\Pi^k_{12}(\jw)-\Pi^k_{12}(\jw)^*+ 2 \Pi^k_{22}(\jw)    \leq 0$.  \\[1ex]
On the other hand, conditions~\eqref{eq:eqtStat1} and \eqref{eq:eqtStat2} rewrite in their expanded forms as
$\Phi_j(\jw) \geq0$  and  $\Psi_j(\jw) \geq0$  with $\Phi_j(\jw)$ and $\Psi_j(\jw)$ given by~
\begin{equation*}
\hspace{-1.5cm}
\begin{matrix}
&\Phi_j(\jw) =-\tau^2\epsilon_{B_j}(\w) B_j(\jw)^*B_j(\jw)+\cdots\\&\cdots+	\begin{pmatrix}\tau B_j(\jw) \\ 1  \end{pmatrix}^*	
\Big( \gamma_{j1}(\w)\Pi^{B_j}_1+\cdots\\ &~~~~~~~~~~~~\cdots+ \gamma_{j2}(\w)\Pi^{B_j}_2 (\jw)\Big)
\begin{pmatrix}\tau B_j(\jw) \\1  \end{pmatrix}
\end{matrix}
\end{equation*}  
and		\begin{equation*}
\begin{matrix}
&\Psi_j(\jw) =-\tau^2\epsilon_{G_j}(\w) G_j(\jw)^*G_j(\jw)+\cdots\\&\cdots+	\begin{pmatrix}\tau G_j(\jw)\\I  \end{pmatrix}^*  	\Big( \delta_{j1}(\w)\Pi^{G_j}_1+\cdots\\ &~~~~~~~~~~~~\cdots+\delta_{j2}(\w) \Pi^{G_j}_2+\delta_{j3}(\w)\Pi^{G_j}_3 (\jw)\Big)
\begin{pmatrix} \tau G_j(\jw)\\I  \end{pmatrix} 
\end{matrix}
\end{equation*}

For $\tau=1$,  $\Phi_j(\jw) \geq0$  and  $\Psi_j(\jw) \geq0$ are satisfied from \textbf{Part~1} and \textbf{Part~2}.\\
For $\tau=0$,  $\Phi_j(\jw) \geq0$  and  $\Psi_j(\jw) \geq0$ are also  satisfied as   %{$$\sum_{i=1}^{2}\gamma_{ji}(\w)\big( \Pi^{B_j}_i (\jw)\big)_{22}\geq0$$} 
\kkl{$\gamma_{j1}(\w)\big( \Pi^{B_j}_1\big)_{22}+\gamma_{j2}(\w)\big( \Pi^{B_j}_2 (\jw)\big)_{22}\geq0$}
~~and ~~
\kkl{$\delta_{j1}(\w)\big( \Pi^{G_j}_1\big)_{22}+\delta_{j2}(\w)\big( \Pi^{G_j}_2\big)_{22}+\delta_{j3}(\w)\big( \Pi^{G_j}_i (\jw)\big)_{22}\geq0$}	 
%$$ \sum_{i=1}^{3}\delta_{ji}(\w)\big( \Pi^{G_j}_i (\jw)\big)_{22}\geq0$$
where $\big( \Pi^{B_j}_i\big)_{22}$ and $\big( \Pi^{G_j}_3 \big)_{22}$ are the lower right blocks of  $\Pi^{B_j}_i (\jw)$ and $\Pi^{G_j}_i (\jw)$ respectively.\\
For $\tau\in(0,1)$,  $\Phi_j(\jw) \geq0$  and  $\Psi_j(\jw) \geq0$ are also satisfied   since $\Phi_j(\jw)$ and $\Psi_j(\jw)$ are concave in~$\tau$ as 
\kkl{$\gamma_{j1}(\w)\big( \Pi^{B_j}_1\big)_{11}+\gamma_{j2}(\w)\big( \Pi^{B_j}_2 (\jw)\big)_{11}\leq0$} and   
\kkl{$\delta_{j1}(\w)\big( \Pi^{G_j}_1\big)_{11}+\delta_{j2}(\w)\big( \Pi^{G_j}_2\big)_{11}+\delta_{j3}(\w)\big( \Pi^{G_j}_3 (\jw)\big)_{11}\leq~0$}.
Therefore, {condition} \eqref{eq:eqtStat1} together with~\eqref{eq:ALA_QC_Con} and {condition}~\eqref{eq:eqtStat2} together with~\eqref{eq:A_QC_Con}   remain satisfied when using the aforementioned homotopies, {and hence} the \mbox{point -1 } remains not included in the corresponding {eigenloci of the return-ratio.} Moreover, using the result of \textbf{Part~3},  if either Statement~1 or Statement~2 are satisfied {at each frequency} then the point $-1$ remains not included in the {eigenloci of the return-ratio of} interconnection~\eqref{eq:bipartite_eqt} and {hence} the winding number of  the point $-1$ does not~change {throughout the homotopies described above}.  {Therefore, since the winding number is zero for $\tau=0$, it is also zero for $\tau=1$.}

To summarize, 	if either Statement~1 or Statement~2 holds  at each frequency ${\omega \in \overline{\R}_+}$ then the eigenloci of $ B(\jw) \mathcal{A}  L(\jw)  \mathcal{A}^{\top} $  do not include  the point -1 and do not encircle it. Hence,  from  the  multivariable Nyquist criterion~\cite{DeY:80}  it follows that the closed-loop transfer functions  of interconnection~\eqref{eq:bipartite_eqt} have no poles in~$\overline{\C}_+$.  Therefore,   %using \cite[Thm 3.7]{Kha:92},
the equilibrium~\eqref{eq:equilibrium} of the power system~\eqref{eq:Line_dynamics}-\eqref{eq:bus_dynmics} with its  small-signal model~\eqref{eq:bipartite_eqt} is locally asymptotically stable which concludes the proof of Theorem~\ref{thm:Main_result_1}.  $\qed$

\section{Proof of Theorem~\ref{thm:Main_result_2}}	\label{app:Thm2_proof}	
\kl{The proof is similar to the proof of Theorem~\ref{thm:Main_result_1}, however the
presence of \ilff{an} integrator  introduces additional complications in the analysis that need to be explicitly addressed.  }

\kkl{Consider the small-signal model~\eqref{eq:bipartite_eqtNN} and consider the following decomposition of $L(s)$ \ilc{and $\Theta_L(s)$}
$$L(s)=\dfrac{1}{s} H(s)~~~~~~~~~~\Theta_L(s)=\dfrac{1}{s} \Theta_H(s).$$
The voltage and the current deviations $	\overline{V_B}(s)$ and $\overline{I_L}(s)$ can be written in terms of the initial conditions $\overline{x_B}(0)$ and  $\overline{x_L}(0)$  as	
\begin{equation}
\begin{pmatrix}
\overline{V_B}(s)\\\overline{I_L}(s)
\end{pmatrix}= \chi(s)
\begin{pmatrix}	\overline{x_B}(0)\\ \overline{x_L}(0) \end{pmatrix}  
\label{eq:LaplaceOfImpulseResponse_N}
\end{equation}
with
\begin{equation*}
\begin{split}
&\chi_{11}(s)=\left(I+B(s)\mathcal{A} s^{-1}H(s)\mathcal{A}^{\top} \right)^{-1} \Theta_B(s)\\
&\chi_{12}(s)=-\left(I+B(s)\mathcal{A} s^{-1}H(s)\mathcal{A}^{\top} \right)^{-1}B(s)\mathcal{A}~s^{-1}\Theta_H(s)\\
&\chi_{21}(s)=\left(I+ s^{-1}H(s)\mathcal{A}^{\top}B(s)\mathcal{A} \right)^{-1}s^{-1}H(s)\mathcal{A}^{T}~\Theta_B(s) \\
&\chi_{22}(s)=\left(I+ s^{-1}H(s)\mathcal{A}^{\top}B(s)\mathcal{A} \right)^{-1}s^{-1} \Theta_H(s)
\end{split}
\end{equation*}}\kkl{with  $\Theta_B(s)$ given  by~\eqref{eq:Theta_B}  and $\Theta_H(s)=s~\Theta_L(s)$ with $\Theta_L(s)$  given  by~\eqref{eq:Theta_L}.  Note that  $B(s)$, $\Theta_B(s)$, $H(s)$ and $\Theta_H(s)$ have no poles in~$\overline{\C}_+$ from Assumption~\ref{assump:stab_bus} and Assumption~\ref{assump:line_general_dynamic}}.

\kl{The proof of Theorem~\ref{thm:Main_result_2} \ilff{has two  parts}.  We show in  {\bf Part~1}
that the function~$\chi(s)$ has at most one  pole  at $s=0$  \ilff{using ideas
analogous to those in~\cite{DKAL:17}}, and we show in  {\bf Part~2}
that $\chi(s)$ has no poles in the closed right half-plane excluding the origin \ie~$\overline{\C}_+\setminus\{0\}$. \kkl{These
results are   used     to deduce the convergence of $\overline{v_B}(t)$ and~$\overline{i_L}(t)$
to a \ilc{constant value}.}}\\[1ex] 
\textbf{$\bullet$ Part~1:} We show in this part that $\chi(s)$ in~\eqref{eq:LaplaceOfImpulseResponse_N} has  one simple pole  at 0.  This is done in two steps.
\begin{itemize}
\item \textit{Step~1: } \kl{We show   that $\left(I+B(s)\mathcal{A} s^{-1}H(s)\mathcal{A}^{\top} \right)^{-1}$ and $\left(I+ s^{-1}H(s)\mathcal{A}^{\top}B(s)\mathcal{A} \right)^{-1}$ have no poles at~$s=0$. To do this, 
we  start by showing that $\left(I+B(s)\mathcal{A} s^{-1}H(s)\mathcal{A}^{\top} \right)$    has no zeros at $s=0$.} \\     
\kl{Note that $ s\left(I+B(s)\mathcal{A} s^{-1}H(s)\mathcal{A}^{\top} \right)$ is equal to $ \left(sI+B(s)\mathcal{A} H(s)\mathcal{A}^{\top} \right)$. Since the underlying graph is \ic{connected} 
then $\mathcal{A}H(\j 0)  \mathcal{A}^{\top}$  has a  simple eigenvalue at \ic{the origin.} %$s=0$.
Using the fact that  $B(\j0)=\oplus_{j=1}^{n_b}B_j(\j0)>0$ \ic{(see Remark~\ref{rem:passivity_lowFreq})} we have that  $ s\left(I+B(s)\mathcal{A} s^{-1}H(s)\mathcal{A}^{\top} \right)$  has a simple zero  at $s=0$, hence $\left(I+B(s)\mathcal{A} s^{-1}H(s)\mathcal{A}^{\top} \right)$ has no zeros at $s=0$ and $\left(I+B(s)\mathcal{A} s^{-1}H(s)\mathcal{A}^{\top} \right)^{-1}$ has no poles at $s=0$.    The reasoning is  the  same to show that~$\left(I+ s^{-1}H(s)\mathcal{A}^{\top}B(s)\mathcal{A} \right)^{-1}$ has no poles at $s=0$. } \\
\item \textit{Step~2: } \kl{We show in this step that $\chi_{12}(s)$ and  $\chi_{21}(s)$ have no poles at $s=0$.  	
For this purpose, we use the following limit from~\cite{CaM:09}: for a complex  non-square matrix $\Lambda$, the limit $\lim_{s\rightarrow 0}(sI+\Lambda^*\Lambda)^{-1}\Lambda^*$ is equal to the pseudo-inverse of $\Lambda$. We start  with $\lim_{s\rightarrow 0} \chi_{12}(s)$. \kkl{\ilc{We recall} that $B(s)$, $H(s)$ and $\Theta_H(s)$ have no poles at $s=0$ and \ilc{note} that\footnote{\kkl{This can be deduced from Assumption~\ref{assump:line_general_dynamic} and the decomposition of $L(s)=s^{-1}H(s)$.}} $H(\j0)>0$.   Therefore,   $\lim_{s\rightarrow 0}\chi_{12}(s)= \lim_{s\rightarrow 0} U (sI+V^*V)^{-1}V^*W^{-1} \Theta_H(\j0)$ with $U=\left( B(\j0)\right)^{\frac{1}{2}}$, $W= \left(H(\j0)\right)^{\frac{1}{2}}$ and $V^*=U\mathcal{A} W$;  \ilff{we thus see} that $\chi_{12}(s)$ exists at $s=0$ and hence $\chi_{12}(s)$ has no poles at $s=0$.}
Using a similar argument, we can show that  $\chi_{21}(s)$ has no poles at $s=0$.}  	 
\end{itemize}

\kl{ Therefore, \kkl{from the  previous  steps}, we conclude that $\chi_{11}(s)$,   $\chi_{12}(s)$ and  $\chi_{21}(s)$  have no poles at $s=0$ while  $\chi_{22}(s)$ has one simple pole at $s=0$ and hence the function~$\chi(s)$ in~\eqref{eq:LaplaceOfImpulseResponse_N} has one simple pole at $s=0$.}\\[1ex]

\kl{\textbf{$\bullet$ Part~2:} We show in this part that if at least one of the statements of Theorem~\ref{thm:Main_result_1} is satisfied for  \icl{each}~${\omega \in \overline{\R}_+\setminus\{0\}}$  and if $B_j(\j0)>0$, then	$\chi(s)$ of~\eqref{eq:LaplaceOfImpulseResponse_N} has no poles \icl{in
%\footnote{\icl{Recall that}  $\Theta_B(s)$ and $\Theta_H(s)$ have no poles in~$\overline{\C}_+$.}
$\overline{\C}_+\setminus\{0\}$.}
For this purpose, we adopt \kkl{the \ilc{notation below that is} used to} 
define a modified  Nyquist contour.
\begin{itemize}
\item  $C_R$, with \icl{$R>0$ sufficiently large,} %$R\rightarrow\infty$,
is the semi-circle centered at the origin  of radius $R$ in the  right-half plane, that is
$C_R=\{ s \in  \C :  \left| s \right| =R,   \text{Re}(s)\geq0\}$ where $\text{Re}(s)$ denotes the real part of $s$.
\item $c_r (\j 0)$, with \icl{$r>0$ sufficiently small,} %$r\rightarrow 0$,
\mbox{is the semi-circle} centered at the origin  of radius $r$  in the  right-half plane, that~is {$c_r (\j 0)=\{ s \in \C : \left| s\right| =r,   \text{Re}(s)>0\}$}.
\item $C_{\ell\setminus r}$ is the  contour parameterized by \icl{$r$} %$r\rightarrow 0$
defined by $C_{\ell\setminus r}=\j(-\infty,-r] \cup \j[+r,\infty) $ 		which is a straight line on the imaginary axis excluding the segment $\j(-r,+r)$.
\end{itemize}
The modified Nyquist contour is given by
\begin{equation}
C_N= C_{\ell\setminus r} \cup C_R \cup     c_r (\j 0) .
\label{eq:Modified_Nyquist_Countour}
\end{equation}}		
\kl{As    $B(s)\mathcal{A} s^{-1}H(s)\mathcal{A}^{\top}$ and $s^{-1}H(s)\mathcal{A}^{\top}B(s)\mathcal{A}$  have the same nonzero eigenvalues,    $\left(I+B(s)\mathcal{A} s^{-1}H(s)\mathcal{A}^{\top} \right)^{-1}$ has the same poles as $\left(I+ s^{-1}H(s)\mathcal{A}^{\top}B(s)\mathcal{A} \right)^{-1}$. Therefore,
to show that $\chi(s)$ of~\eqref{eq:LaplaceOfImpulseResponse_N} has no poles in $\overline{\C}_+\setminus\{0\}$, \ilff{we only need to show that  $\left(I+B(s)\mathcal{A} s^{-1}H(s)\mathcal{A}^{\top} \right)^{-1}$   has no poles in~$\overline{\C}_+\setminus\{0\}$,  %\kkl{
since $B(s)$, $H(s)$, $\Theta_B(s)$ and $\Theta_H(s)$ have no poles in~$\overline{\C}_+$}. To do this, we define the return-ratio
\begin{equation*}
Q(s)=B(s)\mathcal{A} s^{-1}H(s)\mathcal{A}^{\top}\label{eq:OCRRN}
\end{equation*}
and we show that the \mbox{point~-1} is not included in the eigenloci of $Q(s)$ when $Q(s)$ is evaluated  along  the modified Nyquist contour~$C_N$  in~\eqref{eq:Modified_Nyquist_Countour}. \ic{Furthermore, we show that the point $-1$ is not included in the eigenloci when a linear homotopy is carried out from the origin.} This is done in three steps.}
\begin{itemize}
\item \textit{Step~1: } \kl{We   consider the contribution of the straight   line  $C_{\ell\setminus r}$  to the eigenloci. This portion  is included in $\jR\setminus\{\j  0\}$. %We exploit
\icl{Due to} the symmetry of  the eigenloci  about the real axis, \icl{it is sufficient to evaluate  $Q(s)$  over $\jR_+\setminus\{\j  0\}$.}} \\
\kl{Using \icl{arguments \ic{analogous} to those} in  \textbf{Part~1},  \textbf{Part~2} and  \textbf{Part~3} of  the proof of Theorem~\ref{thm:Main_result_1}, we   \ic{deduce}
that if  at least one of the statements of Theorem~\ref{thm:Main_result_1} is satisfied  \ic{for each}~${\omega \in \overline{\R}_+\setminus\{0\}}$ then  the \mbox{point~-1} is not included in the eigenloci of the return-ratio $Q(s)$ when evaluated on $C_{\ell\setminus r}$.		
We now consider  the linear homotopy where each $B_j(s)$ is replaced by  $\tau B_j(s)$ with $\tau \in[0,1]$. Using \ic{arguments analogous to those} in  \textbf{Part~4} in  the proof  of  Theorem~\ref{thm:Main_result_1}, we \ic{deduce} that  if either Statement~1 or Statement~2 are satisfied \ic{at each} frequency ${\omega \in \overline{\R}_+\setminus\{0\}}$, then the \mbox{point~-1} remains not included in the eigenloci of $Q(s)$ as $\tau$ changes continuously in~$[0,1]$.}\\

\item \textit{Step~2: }\kl{
We consider  now the contribution  of  \ic{$C_R$
to the eigenloci as $R\rightarrow\infty$}. \ic{Since $Q(s)$ is proper, we have
$Q(s)\rightarrow  Q_{\infty}$, as   $R \rightarrow \infty$,  where $Q_{\infty}=Q(\j\infty)$ is a constant matrix. Hence,} the  eigenvalues   of $Q(s)$  evaluated along  $C_R$
tend to constant  points equal to the eigenvalues of~$Q_{\infty}$.
\ic{Therefore,} from \textit{Step~1}, if $-1\notin \lambda_i  (Q(\j\infty))$ then $-1\notin\lambda_i (Q(s))$ for $|s|\geq R_0$ with $R_0$ \ic{a sufficiently large number}.\\ Moreover, \ic{when we consider} %using
the %the aforementioned
homotopy \ic{described in \textit{Step~1}} with $\tau \in[0,1]$, we have that  \ic{$Q_{\infty}=\tau Q(\j\infty)$}
\ic{and the \mbox{point~-1} remains not included in $\lambda_i  (Q_{\infty})$ throughout this homotopy. Hence, $-1\notin\lambda_i (\tau Q(s))$ for $|s|\geq R_0$, for $\tau\in[0,1]$ with $R_0$ a sufficiently large number.}}\\

\item \textit{Step~3: }	\kl{We now investigate   the contribution of the  semi-circle $c_r(\j 0)$   as $r \rightarrow 0$. 	When traversing the semi-circle  $c_r(\j 0)$ corresponding to the pole at $s=0$, 
\ic{each eigenlocus $\lambda_k\left(Q(c_r) \right)$}   
has a magnitude that tends to $\infty$ as $r\rightarrow0$   while its argument    changes  by $\pi$ radians. Moreover,  \kkl{we know that $H(\j0)>0$}  and  that \ic{for each $j$} we have  $B_j(\j0)>0$, whence by continuity of the transfer function there exists sufficiently small~$\mu$ such that $\frac{B_j(\j\nu)}{\nu}>0$  and $\mathcal{A}H(\j\nu) \mathcal{A}^{\top}\geq 0$ for all \ic{$\nu \in (0,\mu)$.}   Therefore,  the arc corresponding to the eigenlocus along $c_r (\j 0)$, \ic{for $r$ sufficiently small}
is closed through the right half-plane which lies to the right of the \mbox{point~-1}.\\ Now, if we  replace each $B_j(s)$  by  $\tau B_j(s)$ with $\tau \in[0,1]$, \ic{the eigenloci of the return ratio $Q(s)$, when the latter is evaluated along $c_r (\j 0)$, remain in the right half-plane and hence do not include the point~$-1$.}}

\end{itemize} 

\kl{From the previous \ic{steps we have that %and using the aforementioned homotopy,  
if} the conditions of Theorem~\ref{thm:Main_result_2} are satisfied then  the  point -1 is not included in the eigenloci of $Q(s)$ when $Q(s)$ is evaluated  along  the modified Nyquist contour~$C_N$ \ic{in}~\eqref{eq:Modified_Nyquist_Countour} with \ic{$R>0$ sufficiently large and $r>0$ sufficiently small.}
%$R\rightarrow\infty$  and $r\rightarrow0$. 
Moreover,  \ic{when the homotopy described in the previous steps is carried out  the eigenloci still do not include the point~$-1$. Therefore,} the winding number of  the point $-1$ does not~change \ic{throughout this homotopy} and remains equal to zero\footnote{Since the winding number is zero for $\tau=0$, it is also zero for $\tau=1$.}.  
Therefore,   from  the  multivariable Nyquist criterion~\cite{DeY:80},  it follows that $\left(I+B(s)\mathcal{A} s^{-1}H(s)\mathcal{A}^{\top} \right)^{-1}$ and $\left(I+ s^{-1}H(s)\mathcal{A}^{\top}B(s)\mathcal{A} \right)^{-1}$  have no poles in
$\overline{\C}_+\setminus\{0\}$ and consequently  $\chi(s)$ in~\eqref{eq:LaplaceOfImpulseResponse_N} has no poles in
$\overline{\C}_+\setminus\{0\}$ as well. }

\kkl{To summarize,  we deduce from \textbf{Part~1} and \textbf{Part~2} that when the conditions {\it C1} and {\it C2}   of Theorem~\ref{thm:Main_result_2} are satisfied then  $\chi(s)$ has no poles in~$\overline{\C}_+$ except   $\chi_{22}(s)$ which has a simple pole at $s=0$. Therefore,  for all $\overline{x_B}(0)$ and   $\overline{x_L}(0)$, we have $\overline{v_B}(t)\rightarrow \overline{v_{B}}_\infty$   as $t\rightarrow \infty$ with ${\overline{v_{B}}_\infty=0}$.
For   $\overline{i_L}(t)$, \ilc{due  to} the simple pole at origin  of $\chi_{22}(s)$, we have  $\overline{i_L}(t)\rightarrow \overline{i_{L}}_\infty$    as $t\rightarrow \infty$ with ${\overline{i_{L}}_\infty}$ some constant  depending on $\overline{x_L}(0)$. Therefore,   for all initial conditions $\overline{x_B}(0)$ and   $\overline{x_L}(0)$,
the voltage and the current deviations    $\overline{v_B}(t)$ and  $\overline{i_L}(t)$  converge to a \ilc{constant %equilibrium
value}, \ilc{which completes the} proof of Theorem~\ref{thm:Main_result_2}.  $\qed$}

\addtolength{\baselineskip}{-2pt}
\bibliographystyle{unsrt}
\balance
\bibliography{KL507_bib}

\begin{thebibliography}{10}

\bibitem{JMLJ:13}
J.~John Justo, F.~Mwasilu, J.~Lee, and J.W. Jung.
\newblock {AC}-microgrids versus {DC}-microgrids with distributed energy
  resources: A review.
\newblock {\em Renewable and Sustainable Energy Reviews}, 24:387 -- 405, August
  2013.

\bibitem{EMM:15}
A.~T. Elsayed, A.~A. Mohamed, and O.~A. Mohammed.
\newblock {DC} microgrids and distribution systems: An overview.
\newblock {\em Electric Power Systems Research}, 119:407 -- 417, February 2015.

\bibitem{MSTKFLG:17}
L.~{Meng}, Q.~{Shafiee}, G.~F. {Trecate}, H.~{Karimi}, D.~{Fulwani}, X.~{Lu},
  and J.~M. {Guerrero}.
\newblock Review on control of {DC}~microgrids and multiple microgrid clusters.
\newblock {\em IEEE Journal of Emerging and Selected Topics in Power
  Electronics}, 5(3):928--948, September 2017.

\bibitem{DLVG:16}
T.~{Dragi\v{s}evi\'{c}}, X.~{Lu}, J.~C. {Vasquez}, and J.~M. {Guerrero}.
\newblock {DC} microgrids-part {I}: {A} review of control strategies and
  stabilization techniques.
\newblock {\em IEEE Transactions on Power Electronics}, 31(7):4876--4891, July
  2016.

\bibitem{GVMVC:11}
J.~M. {Guerrero}, J.~C. {Vasquez}, J.~{Matas}, L.~G. {de Vicuna}, and
  M.~{Castilla}.
\newblock Hierarchical control of droop-controlled {AC} and {DC} microgrids-a
  general approach toward standardization.
\newblock {\em IEEE Transactions on Industrial Electronics}, 58(1):158--172,
  January 2011.

\bibitem{ZhD:15}
J.~Zhao and F.~D\"{o}rfler.
\newblock Distributed control and optimization in {DC} microgrids.
\newblock {\em Automatica}, 61:18 -- 26, November 2015.

\bibitem{TRF:18}
M.~Tucci, S.~Riverso, and G.~Ferrari-Trecate.
\newblock Line-{Independent Plug-and-Play Controllers for Voltage
  Stabilization} in {DC} microgrids.
\newblock {\em IEEE Transactions on Control Systems Technology},
  26(3):1115--1123, May 2018.

\bibitem{NMDL:15}
V.~{Nasirian}, S.~{Moayedi}, A.~{Davoudi}, and F.~L. {Lewis}.
\newblock Distributed cooperative control of dc microgrids.
\newblock {\em IEEE Transactions on Power Electronics}, 30(4):2288--2303, April
  2015.

\bibitem{GLH:15}
Y.~{Gu}, W.~{Li}, and X.~{He}.
\newblock Passivity-based control of {DC} microgrid for self-disciplined
  stabilization.
\newblock {\em IEEE Transactions on Power Systems}, 30(5):2623--2632, September
  2015.

\bibitem{FCS:21}
Joel Ferguson, Michele Cucuzzella, and Jacquelien M.~A. Scherpen.
\newblock Exponential stability and local iss for dc networks.
\newblock {\em IEEE Control Systems Letters}, 5(3):893--898, 2021.

\bibitem{IDSBGP:18}
A.~{Iovine}, G.~{Damm}, E.~D. {Santis}, M.~D. {Di Benedetto}, L.~{Galai-Dol},
  and P.~{Pepe}.
\newblock Voltage stabilization in a {DC} microgrid by an {ISS}-like {L}yapunov
  function implementing droop control.
\newblock In {\em 2018 European Control Conference (ECC)}, pages 1130--1135,
  2018.

\bibitem{RMOHMM:18}
T.~K. {Roy}, M.~A. {Mahmud}, A.~M.~T. {Oo}, M.~E. {Haque}, K.~M. {Muttaqi}, and
  N.~{Mendis}.
\newblock Nonlinear adaptive backstepping controller design for islanded {DC
  }microgrids.
\newblock {\em IEEE Transactions on Industry Applications}, 54(3):2857--2873,
  June 2018.

\bibitem{CTDPCFVDS:19}
Michele Cucuzzella, Sebastian Trip, Claudio De~Persis, Xiaodong Cheng,
  Antonella Ferrara, and Arjan van~der Schaft.
\newblock A robust consensus algorithm for current sharing and voltage
  regulation in dc microgrids.
\newblock {\em IEEE Transactions on Control Systems Technology},
  27(4):1583--1595, 2019.

\bibitem{PWD:18}
C.~{De~Persis}, E.~R.A. Weitenberg, and F.~D\"{o}rfler.
\newblock A power consensus algorithm for {DC} microgrids.
\newblock {\em Automatica}, 89:364 -- 375, March 2018.

\bibitem{CLKKS:19}
M.~{Cucuzzella}, R.~{Lazzari}, Y.~{Kawano}, K.~C. {Kosaraju}, and J.~M.~A.
  {Scherpen}.
\newblock Robust passivity-based control of boost converters in {DC}
  microgrids.
\newblock In {\em 2019 IEEE 58th Conference on Decision and Control (CDC)},
  pages 8435--8440, 2019.

\bibitem{NSMMFT:20}
P.~Nahata, R.~Soloperto, M.~Tucci, A.~Martinelli, and G.~Ferrari-Trecate.
\newblock A passivity-based approach to voltage stabilization in {DC}
  microgrids with {ZIP} loads.
\newblock {\em Automatica}, 113:108770, March 2020.

\bibitem{DoB:13}
F.~D\"{o}rfler and F.~{Bullo}.
\newblock Kron reduction of graphs with applications to electrical networks.
\newblock {\em IEEE Transactions on Circuits and Systems I: Regular Papers},
  60(1):150--163, Jan 2013.

\bibitem{WOLL:20}
J.~Watson, Y.~Ojo, K.~Laib, and I.~Lestas.
\newblock A scalable control design for grid-forming inverters in microgrids.
\newblock {\em IEEE Transactions on Smart Grid}, 12(6):4726--4739, 2021.
\newblock (to appear).

\bibitem{gar:18}
A.~Garc\'es.
\newblock On the convergence of {N}ewton's method in power flow studies for
  {DC} microgrids.
\newblock {\em IEEE Transactions on Power Systems}, 33(5):5770--5777, 2018.

\bibitem{mon:20}
O.~D. Montoya.
\newblock On the existence of the power flow solution in {DC} grids with {CPL}s
  through a graph-based method.
\newblock {\em IEEE Transactions on Circuits and Systems II: Express Briefs},
  67(8):1434--1438, 2020.

\bibitem{Les:11}
I.~Lestas.
\newblock On network stability, graph separation, interconnection structure and
  convex shells.
\newblock In {\em 2011 IEEE Conference on Decision and Control (CDC)}, pages
  4257--4263, 2011.

\bibitem{LWL:20}
K.~Laib, J.~Watson, and I.~Lestas.
\newblock Decentralized stability conditions for inverter-based microgrids.
\newblock In {\em 2020 IEEE Conference on Decision and Control (CDC)}, pages
  1341--1346, 2020.

\bibitem{iwasaki2005generalized}
Tetsuya Iwasaki and Shinji Hara.
\newblock Generalized {KYP} lemma: Unified frequency domain inequalities with
  design applications.
\newblock {\em IEEE Transactions on Automatic Control}, 50(1):41--59, January
  2005.

\bibitem{JLADSCA:93}
B.~K. Johnson, R.~H. Lasseter, F.~L. Alvarado, D.~M. Divan, H.~Singh, M.~C.
  Chandorkar, and R.~Adapa.
\newblock High-temperature superconducting {DC} networks.
\newblock {\em IEEE Transactions on Applied Superconductivity}, 4(3):115--120,
  September 1994.

\bibitem{BNDL:17}
A.~Bidram, V.~Nasirian, A.~Davoudi, and \mbox{F.~L.~Lewis}.
\newblock {\em Cooperative Synchronization in Distributed Microgrid Control}.
\newblock Springer, 2017.

\bibitem{CRMFM:18}
M.~Cupelli, A.~Riccobono, M.~Mirz, M.~Ferdowsi, and A.~Monti.
\newblock In {\em Modern Control of {DC}-Based Power Systems}. Academic Press,
  2018.

\bibitem{LSGVHW:15}
X.~{Lu}, K.~{Sun}, J.~M. {Guerrero}, J.~C. {Vasquez}, L.~{Huang}, and
  J.~{Wang}.
\newblock Stability enhancement based on virtual impedance for {DC} microgrids
  with constant power loads.
\newblock {\em IEEE Transactions on Smart Grid}, 6(6):2770--2783, November
  2015.

\bibitem{ZSAM:16}
D.~{Zammit}, C.~S. {Staines}, M.~{Apap}, and A.~{Micallef}.
\newblock Paralleling of buck converters for {DC} microgrid operation.
\newblock In {\em 2016 International Conference on Control, Decision and
  Information Technologies (CoDIT)}, pages 070--076, April 2016.

\bibitem{ZDG:95}
K.~Zhou, J.C. Doyle, and K.~Glover.
\newblock {\em Robust and Optimal Control}.
\newblock Prentice Hall, New Jersey, 1995.

\bibitem{Kha:92}
H.K. Khalil.
\newblock {\em Nonlinear Systems}.
\newblock Macmillan, New York, 1992.

\bibitem{DeY:80}
C.~{Desoer} and {Yung-Terng Wang}.
\newblock On the generalized {N}yquist stability criterion.
\newblock {\em IEEE Transactions on Automatic Control}, 25(2):187--196, April
  1980.

\bibitem{DKAL:17}
E.~Devane, A.~Kasis, M.~Antoniou, and I.~Lestas.
\newblock Primary frequency regulation with load-side participation-part {II}:
  {B}eyond passivity approaches.
\newblock {\em IEEE Transactions on Power Systems}, 32(5):3519--3528, September
  2017.

\bibitem{CaM:09}
S.~L. Campbell and C.~D Meyer.
\newblock {\em Generalized Inverses of LinearTransformations}.
\newblock {SIAM}, Classics in Applied Mathematics, 2009.

\end{thebibliography}
\addtolength{\baselineskip}{2pt}

\end{document}